\documentclass{amsart}
\usepackage[utf8]{inputenc}
\usepackage{color}

\usepackage[margin=1.0in]{geometry}

\usepackage{amsmath}
\usepackage{amsthm}
\usepackage{amssymb}
\usepackage[matrix,arrow,curve]{xy}
\usepackage{enumitem} 
\usepackage{graphicx}
\usepackage{mdwlist}	
\usepackage{mathrsfs}\usepackage{longtable}
\usepackage{tikz}
\usetikzlibrary{patterns}
\usepackage{quiver}
\usepackage{caption}
\usepackage{subcaption}
\usepackage{makecell}
\usepackage{ textcomp }
\usepackage{longtable}

\newcommand\setItemnumber[1]{\setcounter{enum\romannumeral\@enumdepth}{\numexpr#1-1\relax}}

\newcommand{\p}{\mathbb{P}}

\DeclareMathOperator{\Id}{Id}
\DeclareMathOperator{\Hom}{Hom}

\DeclareMathOperator{\Gr}{Gr}
\DeclareMathOperator{\ZZ}{\mathbb{Z}}

\DeclareMathOperator{\CC}{\mathbb{C}}
\DeclareMathOperator{\Pic}{Pic}

\DeclareMathOperator{\Sing}{Sing}

\DeclareMathOperator{\Coker}{Coker}

\DeclareMathOperator{\Ker}{Ker}
\renewcommand{\Im}{\mathrm{Im}}

\DeclareMathOperator{\U}{\mathcal{U}}
\DeclareMathOperator{\Spin}{\mathcal{S}}
\DeclareMathOperator{\Spinmap}{s}
\DeclareMathOperator{\OP}{\mathcal{O}}
\DeclareMathOperator{\rk}{rk}

\newcommand{\LA}{\mathcal{A}}
\newcommand{\LB}{\mathcal{B}}
\newcommand{\I}{\mathcal{I}}
\newcommand{\C}{\mathcal{C}}

\newcommand{\B}{\widetilde{B}}

\newcommand{\F}{\mathcal{F}}

\newcommand{\E}{\mathcal{E}}
\newcommand{\K}{\mathcal{K}}
\newcommand{\M}{\mathcal{M}}
\newcommand{\V}{\mathcal{V}}
\newcommand{\doublecover}{M}
\newcommand{\codes}{\mathrm{\overline{C}}_B}

\newcommand{\HomSh}{\mathscr{H}\mathit{om}}
\newcommand{\ExtSh}{\mathscr{E}\mathit{xt}}
\newcommand{\deltamap}{\mathrm{d}}
\DeclareMathOperator{\Ext}{Ext}

\newcommand{\xdashrightarrow}[2][]{\ext@arrow 0359\rightarrowfill@@{#1}{#2}}

\newtheorem{theorem}[equation]{Theorem}
\newtheorem{lemma}[equation]{Lemma}
\newtheorem{proposition}[equation]{Proposition}
\newtheorem{corollary}[equation]{Corollary}

\theoremstyle{definition}
\newtheorem{defin}[equation]{Definition}

\newtheorem{example}[equation]{Example}

\makeatletter
\@addtoreset{equation}{section}
\makeatother

\title[Double quadrics with Artin--Mumford obstructions to rationality]{Double covers of smooth quadric threefolds with Artin--Mumford obstructions to rationality}
\author{Alexandra Kuznetsova}

\address{Steklov Mathematical Institute of Russian Academy of Sciences, Moscow, Russia
}

 \email{sasha.kuznetsova.57@gmail.com}
 \thanks{This work was supported by the Russian Science Foundation under grant no. 23-11-00033, https://rscf.ru/en/project/23-11-00033/}

\begin{document}
\begin{abstract}
 We study obstructions to rationality on a nodal Fano threefold $\doublecover$ that is a double cover of a smooth quadric threefold ramified over an intersection with a quartic threefold in $\mathbb{P}^4$. We prove that if $\doublecover$ admits an Artin--Mumford obstruction to rationality then it lies in one of three explicitly described families. Conversely, a general element of any of these families admits an Artin--Mumford obstruction to rationality. Only one of these three families was known before; other two families of nodal Fano threefolds with obstructions to rationality are new.
\end{abstract}

 \maketitle
 
 \section{Introduction}
 If $X$ is a smooth projective complex threefold and $Y$ is a blow-up of $X$ in a smooth subvariety in $X$ then it is easy to show that the torsion subgroups in the groups $H^3(X,\ZZ)$ and $H^3(Y,\ZZ)$ are isomorphic. This implies that  a non-trivial torsion subgroup in $H^3(X,\ZZ)$ is an obstruction to rationality for a smooth projective complex threefold~$X$. In fact, it is even an obstruction to stable rationality, i.e. the product~\mbox{$X\times \p^m$} is not rational for any $m\geqslant 0$. Artin and Mumford used this fact in their famous paper \cite{Artin_Mumford} in order to construct one of the first examples of a unirational non-rational variety. Note that in dimension 1 and~2 unirational varieties are necessarily rational by L\"uroth  and Castelnuovo theorems respectively. The Artin and Mumford's example shows that these theorems cannot be generalized to higher dimensions.
 
 In their paper Artin and Mumford considered  a nodal quartic double solid $\doublecover$, i.e. a double cover of~$\p^3$ ramified over a quartic surface, choose a resolution $\widetilde{\doublecover}$ of singularities of $\doublecover$ and proved that $H^3(\widetilde{\doublecover},\ZZ)$ contains a non-trivial torsion subgroup. More precisely, the quartic surface in the example of Artin and Mumford is a {\itshape symmetroid}: a zero locus of the determinant of a symmetric $4\times 4$ matrix of linear forms over~$\p^3$.
 
 Endrass in \cite{Endrass99} studied all nodal quartic double solids $\doublecover$ which admit an {\itshape Artin--Mumford obstruction to rationality}, i.e. varieties $\doublecover$ such that the group $H^3(\widetilde{\doublecover},\ZZ)$ contains a non-trivial torsion subgroup where~$\widetilde{\doublecover}$ is a resolution of singularities of $\doublecover$. He proved that the only family of such threefolds is the one constructed by Artin and Mumford. Sextic double solids are very similar to quartic double solids in many aspects. The question whether some of them admit obstructions to rationality is also interesting. These varieties were studied in \cite{IKP}, \cite{ex-Beau}, and \cite{Kuznetsova_sextic_double_solids}, where four families of nodal sextic double solids with Artin--Mumford obstructions to rationality were constructed. It was also shown that any nodal sextic double solid with such obstruction lies in one of these families. 
 
 Our goal in this paper is to construct new examples of stably non-rational varieties with Artin--Mumford obstruction to rationality. 
 Examples of quartic and sextic nodal double solids with obstructions to rationality of this type were obtained using the approach  developed by Endrass in \cite{Endrass99}. We formulate the method in generality necessary for applications in this paper.
 
 Let $X$ be a smooth threefold such that integral cohomology groups of $X$ are torsion-free and $H^3(X,\ZZ) = 0$, let $\pi\colon \doublecover\to X$ be the double cover ramified over a nodal surface $B\subset X$ and let $\widetilde{\doublecover}$ be a resolution of singularities of $\doublecover$. Endrass suggested to study special subsets of singular points on $B$ called~{\itshape\mbox{$\frac{\delta}{2}$-ev}\-en sets of nodes} where~$\delta  = 0$ or $1$, see Definition \ref{definition: even sets of nodes}. Such sets of nodes have nice geometric properties; they were first introduced in \cite{Catanese} and \cite{Beauville} and then studied in \cite{Casnati_Catanese}, \cite{Endrass98}, \cite{Catanese_Tonoli}, \cite{Catanese_Cho_Coughlan_Frapporti_Verra_Kiermaier}. The set of ~\mbox{$\frac{\delta}{2}$-ev}\-en sets of nodes on $B$ has a structure of a finite group and one can construct an embedding of the torsion subgroup~$T\subset H^3(\widetilde{\doublecover},\ZZ)$ into this group, see \cite[Lemma 1.2]{Endrass99}. Thus, if $\doublecover$ admits an Artin--Mumford obstruction to rationality then the surface $B$ contains a $\frac{\delta}{2}$-even set of nodes. A useful construction of a nodal surface in $X$ with a $\frac{\delta}{2}$-even set of nodes was suggested by Barth.
 
 \begin{example}[{\cite{Barth}}]\label{example: Phi}
  Let $X$ be a smooth threefold and let~$\OP_X(1)$ be an ample line bundle on $X$. Let $\E$ be a vector bundle on $X$ and let $\Phi$ be an element of the group~$H^0(X,(S^2\E)(d+\delta))\subset \Hom_{X}(\E^{\vee}(-d-\delta), \E)$ which induces an embedding of sheaves $\Phi\colon \E^{\vee}(-d-\delta)\hookrightarrow \E$. Here $d$ is a positive integer and $\delta = 0$ or $1$.
  Then by \cite{Barth} under a  certain condition on the degree of $\E$ and on $\Phi$ we can define a scheme $B(\Phi)$ of degree $d$ in $X$ and a subscheme $w(\Phi)$ on~$B(\Phi)$ in the following way:
 \begin{align}\label{eq: barth construction}
  B(\Phi) = \left\{\det\left(\Phi\colon \E^{\vee}(-d-\delta)\to \E\right) = 0\right\}; && w(\Phi) = \left\{\mathrm{corank}\left(\Phi\colon \E^{\vee}(-d-\delta)\to \E\right) = 2\right\}.
 \end{align}
 Moreover, if the vector bundle $(S^2\E)(d+\delta)$ is globally generated then there exists a non-empty Zariski open subset $\mathcal{U} \subset  H^0(X,(S^2\E)(d+\delta))$ such that for any $\Phi\in \mathcal{U}$ the scheme $B(\Phi)$ is a nodal surface of degree $d$, the subscheme $w(\Phi)$ is a $\frac{\delta}{2}$-set of nodes on $B(\Phi)$ and the singular locus of $B(\Phi)$ equals $w(\Phi)$.
 
 Note that in the case where $X$ is a projective space, $d = 4$, $\delta = 1$ and $\E = \OP_{\p^3}(-2)^{\oplus 4}$  this construction describes a quartic symmetroid as in Artin and Mumford's example of a nodal quartic double solid with an obstruction to rationality.
 \end{example}
 Vice versa, Casnati and Catanese proved that if $X=\p^3$ and~$B$ is a nodal surface in $\p^3$ of degree $d$ with a~\mbox{$\frac{\delta}{2}$-ev}\-en set of nodes~$w\subset B$ then there exists a vector bundle~$\E$ on $\p^3$ and an element $\Phi\in H^0(\p^3,(S^2\E)(d+\delta))$ such that~\mbox{$B = B(\Phi)$} and $w = w(\Phi)$ as in \eqref{eq: barth construction}, see \cite{Casnati_Catanese} and \cite{Casnati_Catanese_erratum}.
 
 The Casnati and Catanese theorem was used in \cite{Kuznetsova_sextic_double_solids} for the classification of nodal sextic double solids with Artin--Mumford obstructions to rationality. The classification proceeds as follows. Given a nodal double solid $\doublecover$ we consider its branch divisor $B$ which is a nodal surface in $\p^3$ of degree $6$. Since $M$ admits the Artin--Mumford obstruction to rationality, the surface~$B$ contains a $\frac{\delta}{2}$-even set of nodes $w$. Using the Casnati and Catanese theorem  we construct a vector bundle~$\E$ on $\p^3$ and an element $\Phi\in H^0(\p^3,(S^2\E)(d+\delta))$ such that $B = B(\Phi)$ and $w = w(\Phi)$. Then in terms of $\E$ and $\Phi$ we can describe in which situations $\doublecover$ admits an Artin--Mumford obstruction to rationality. In this paper we are going to use the same method in order to study obstructions to rationality for a new class of threefolds.
 
 Let $Q\subset \p^4$ be a smooth complex quadric threefold and let $B\subset Q$ be a nodal surface which is a complete intersection of $Q$ with a quartic hypersurface in $\p^4$. Consider the double cover $\pi\colon \doublecover\to Q$ ramified over the surface $B$. Then $\doublecover$ is a nodal Fano threefold of index $1$ and genus $3$. An example of a Fano threefold of this type with 20 nodes and an Artin--Mumford obstruction to rationality was introduced in~\mbox{\cite[Proposition 5.8]{Przyjalkowski_Shramov}}; see also \cite[Remark 4.3]{IKP}. Our goal is to construct new examples of such varieties.  
 
 We begin with the study of $\frac{\delta}{2}$-even sets of nodes on the complete intersection  $B$ of a smooth quadric  threefold $Q$  and a quartic threefold in $\p^4$. Our first result is the following generalization of the Casnati and Catanese result \cite{Casnati_Catanese} to the case of quadric threefolds. Here by $\Spin$ we denote the {\itshape spinor bundle} on the threefold $Q$, see \cite{Ottaviani-spinors} for the definition, also in Section \ref{subsection: spinor bundles} we briefly recall its properties.

 \begin{theorem}\label{theorem: classification of minimal sets of nodes}
 Let $Q\subset \p^4$ be a smooth quadric threefold and let $B\subset Q$ be a nodal surface which is a complete intersection of $Q$ with a quartic hypersurface in $\p^4$. If~$B$ contains a non-empty $\frac{\delta}{2}$-even set of nodes~$w$, then there exist a $\frac{\delta'}{2}$-even set of nodes $w'$ which maybe differs from $w$, a vector bundle $\E$ on $Q$ and an element $\Phi\in H^0(Q,(S^2\E)(4+\delta))$ such that~$B = B(\Phi)$ and $w' = w(\Phi)$ where $\E$ and $w'$ are as follows:
  \begin{align*}
   \delta = 0:\hspace{1cm} 
               & \text{\textup{(E1)}\hspace{1cm}}|w'| = 16 \text{ and }\E = \OP_Q(-1)^{\oplus 2}\oplus \OP_Q(-2);\\
               & \text{\textup{(E2)}\hspace{1cm}}|w'| = 20 \text{ and }\E = \OP_Q(-1)\oplus \Spin(-1);\\
               & \text{\textup{(E3)}\hspace{1cm}}|w'| = 24 \text{ and }\E = \Spin(-1)^{\oplus 2};\\               
               & \\
   \delta = 1:\hspace{1cm} & \text{\textup{(O1)}\hspace{1cm}}|w'| = 12 \text{ and }\E = \OP_Q(-1)\oplus \OP_Q(-2);\\
               & \text{\textup{(O2)}\hspace{1cm}}|w'| = 20 \text{ and }\E = \OP_Q(-2)^{\oplus 4};\\
               & \text{\textup{(O3)}\hspace{1cm}}|w'| = 20 \text{ and }\E = \Spin(-1).
  \end{align*}
  Moreover, in all these cases there exists a non-empty Zariski open subset $\mathcal{U} \subset  H^0(X,(S^2\E)(d+\delta))$ such that for any $\Phi\in \mathcal{U}$ the surface $B$ is nodal and the singular locus of $B$ equals $w(\Phi)$.
 \end{theorem} 
 The idea of the proof of Casnati and Catanese theorem \cite{Casnati_Catanese} is to associate to a $\frac{\delta}{2}$-even set of nodes~$w$ on a surface $B\subset\p^3$ a sheaf $\F$ and to construct a special locally free resolution of its direct image to $Q$. Then the vector bundle~$\E$ and the map $\Phi$ can be defined in terms of this resolution. In order to construct this resolution Casnati and Catanese use Horrocks theorem in the form proved by Walter in~\mbox{\cite[Theorem 0.4]{Walter}} which claims that if~$\E$ is a vector bundle on the projective space $\p^m$ then there exists a unique decomposition of~$\E$ into a direct sum of a vector bundle $\mathcal{S}yz(\E)$ and several line bundles on~$\p^m$; moreover, the vector bundle~$\mathcal{S}yz(\E)$ depends only on the intermediate cohomology groups of all twists of $\E$ in the sense that it is a unique bundle with no summands of rank 1 and such that~\mbox{$H^i(\p^m,\mathcal{S}yz(\E)(n)) = H^i(\p^m,\E(n))$} for all $n\in \ZZ$ and all $1\leqslant i\leqslant m-1$. 
 
 The proof of Theorem \ref{theorem: classification of minimal sets of nodes} is similar to the one in \cite{Casnati_Catanese}. We consider the surface $B$ with a $\frac{\delta}{2}$-even set of nodes $w$, associate a sheaf $\F$ to this set and get $\E$ and $\Phi$ from a special resolution of the direct image of $\F$ to $Q$. However, there are certain differences in proofs since the Horrocks theorem does not hold for a quadric threefold. One can check that if~\mbox{$\E_1 =  \Omega^1_{\p^4}|_Q\oplus T_{\p^4}(-2)|_Q$} and~$\E_2 = S^2\Spin$, then the intermediate cohomology groups of~$\E_1(n)$ and~$\E_2(n)$ are the same for all $n\in \ZZ$ though none of these vector bundles has summands of rank 1 and they are not isomorphic. Thus, there is no hope to construct a universal vector bundle $\mathcal{S}yz(\E_1)$. The crucial reason why the Walter's proof of the Horrocks criterion can not be extended to a quadric threefold is that the spectrum of the graded algebra which defines it is singular unlike the case of the projective space.
 
 Therefore, in some cases it is quite difficult to show that we can construct a resolution of the necessary type and we need to study properties of $\frac{\delta}{2}$-even sets of nodes on $B$ in detail in order to do it. That is why we study $\frac{\delta}{2}$-even sets of nodes with additional assumption. Namely, we assume that $w$ is {\itshape minimal} i.e. that~$w$ does not contain a proper $\frac{\delta}{2}$-even subset of nodes. 
 
 This assumption implies in particular that $|w|\leqslant 28$, see Corollary \ref{corollary: properties of w}. In Sections \ref{section: 0-even sets} and \ref{section: 1/2-even sets less than 20} we study a minimal $\frac{\delta}{2}$-even set~$w$ of nodes whose cardinality strictly less than $28$ and show that in this case $w = w'$ and $B$ is a surface as in cases \textup{(E1 - E3)} and \textup{(O1 - O3)}. The most difficult part of the theorem is the case of a minimal set of $28$ nodes. We show that a surface with such set of nodes contains also another $\frac{1}{2}$-even set of nodes.
 \begin{proposition}\label{proposition: 28 nodes implies 12 nodes}
  Let $B$ be a nodal complete intersection of a smooth quadric hypersurface and a quartic hypersurface in $\p^4$. If the surface $B$ contains a minimal $\frac{\delta}{2}$-even set $w$ of $28$ nodes, then $\delta = 1$ and ~$B$ contains a~\mbox{$\frac{1}{2}$-ev}\-en set $w'$ of $12$ nodes. Sets $w$ and $w'$ are disjoint. 
 \end{proposition}
  Proposition \ref{proposition: 28 nodes implies 12 nodes} implies that a surface $B$ containing a $\frac{\delta}{2}$-even set of 28 nodes is as in the case \textup{(O1)}. In order to prove it we assume that there exists a nodal surface $B$ with a $\frac{\delta}{2}$-even set of 28 nodes. In this assumption we construct a vector bundle~$\E$ and a section $\Phi\in H^0(B,(S^2\E)(5))$ such that $B = B(\Phi)$ and $w = w(\Phi)$. Then we study properties of the vector bundle $\E$ and prove that there exists a hyperplane section of $Q$ which is tangent to the surface $B$ along a curve $C$. Finally we conclude that in this situation $w' = \Sing(B)\cap C$ is a $\frac{1}{2}$-even set of 12 nodes.  
  
  Note that Proposition \ref{proposition: 28 nodes implies 12 nodes} works only under the assumption that there exist a quartic and smooth quadric hypersurfaces in $\p^4$ such that their intersection is a nodal surface with a minimal $\frac{1}{2}$-even set of 28 nodes. Moreover, in Section \ref{section: 1/2-even sets of 28 nodes} we give a concrete description of such surface if it exists. However, we can not check whether our construction actually gives a nodal surface. It would be very interesting to find out whether a nodal surface with such set of nodes exists.

 The description of complete intersections of quadric and quartic threefolds with minimal $\frac{\delta}{2}$-even sets of nodes given in Theorem \ref{theorem: classification of minimal sets of nodes} allows us to study Artin--Mumford obstructions to rationality of double covers of the smooth quadric threefold ramified over them. This leads us to the main result of this paper.
 \begin{corollary}\label{corollary: obstructions to rationality}
  Let $Q\subset \p^4$ be a smooth quadric hypersurface, let $B$ be the nodal complete intersection of $Q$ with a quartic hypersurface and let $\doublecover$ be the double cover of $Q$ ramified over $B$. If $\doublecover$ admits an Artin--Mumford obstruction to rationality then $B$ is a surface with a minimal $\frac{\delta}{2}$-even set of nodes from one of the cases \textup{(E2)}, \textup{(E3)} or \textup{(O2)} of Theorem \textup{\ref{theorem: classification of minimal sets of nodes}}. 
  
  Moreover, the double cover of $Q$ ramified over a general surface $B(\Phi)$ from one of these three cases admits an Artin--Mumford obstruction to rationality.
 \end{corollary}
 Thus, there are three families of nodal Fano threefolds of index $1$ and genus $3$ with Artin--Mumford obstructions to rationality. The family \textup{(O2)} is the one described in \cite[Proposition 5.8]{Przyjalkowski_Shramov}. The other two examples are new. On the contrary, a general double cover of a smooth quadric ramified over a surface described in the case \textup{(E1)}, \textup{(O1)} or \textup{(O2)} does not admit an Artin--Mumford obstruction to rationality.
 
 By \cite[Section 2.1]{Iskovskikh_Pukhlikov} all smooth double covers of quadric threefolds ramified over their intersections with  quartic hypersurfaces in $\p^4$ are non-rational. In view of \cite{Voisin}, Corollary \ref{corollary: obstructions to rationality} gives a new proof of the fact that a very general Fano threefold of this type is stably non-rational, see also \cite[Remark 5.9]{Przyjalkowski_Shramov}. Moreover, since double covers of this type form a subfamily in the family of Fano threefolds of index $1$ and genus~$3$ whose general element is a quartic threefold we see also that a very general quartic threefold is stably non-rational as it was shown in~\cite{Colliot-Thelene_Pirutka}. 
 
 \medskip

 The paper is organized as follows: in Section \ref{section: preliminaries} we recall some properties of smooth quadric threefolds and vector bundles on them. We also define $\frac{\delta}{2}$-even sets of nodes on complete intersections of quadric and quartic threefolds and study their properties. In the next three sections we prove special cases of Theorem~\ref{theorem: classification of minimal sets of nodes}. In Section \ref{section: 0-even sets} we deal with minimal $0$-even sets of nodes, then in Section  \ref{section: 1/2-even sets less than 20} we study minimal $\frac{1}{2}$-even sets of at most 20 nodes and in Section \ref{section: 1/2-even sets of 28 nodes} we prove Proposition \ref{proposition: 28 nodes implies 12 nodes}. Finally, in Section \ref{section: proofs} we complete the proof of Theorem \ref{theorem: classification of minimal sets of nodes} and prove Corollary \ref{corollary: obstructions to rationality}.
 
 \medskip

{\bf Acknowledgements.} I am very grateful to Costya Shramov for suggesting the problem, his interest in this work and many useful discussions. I also would like to thank anonymous referees for useful comments.

\section{Preliminaries}\label{section: preliminaries}
\subsection{Spinor bundles} \label{subsection: spinor bundles}
Let $Q$ be a smooth quadric hypersurface in $\p^4$. The {\itshape spinor bundle} $\Spin$ on $Q$ is a natural generalization of the notion of the universal bundle on a Grassmannian variety. In order to define it consider the Grassmannian $\Gr(2,4)$, by Pl\"ucker it can be embedded to $\p^5$ and the image is a smooth quadric hypersurface of dimension $4$. Consider $Q$ as a hyperplane section of $\Gr(2,4)$. If we denote by $\U$ the tautological vector bundle on $\Gr(2,4)$, then the spinor bundle $\Spin$ on the quadric $Q$ can be defined as follows:
\begin{equation*}
 \Spin = \U|_{Q}.
\end{equation*}
This definition does not depend on the choice of a smooth hyperplane section of the Grassmannian $\Gr(2,4)$.
The vector bundle $\Spin$ on a smooth quadric threefold satisfies the following properties.
\begin{lemma}[{\cite{Ottaviani-spinors}, \cite{Ottaviani-Horrocks_criterion}}]\label{lemma: properties of Spin}
 Let $Q\subset \p^4$ be a smooth quadric threefold and let~$\Spin$ be the spinor bundle on~$Q$. Then the following assertions hold.
 \begin{enumerate}
  \item[$(1)$] One has $\deg(\Spin) = -1$ and $\Spin^{\vee} \cong \Spin(1)$. Moreover, $\Spin$ fits the following exact sequence: 
  \begin{equation}\label{eq: spin}
     0\to \Spin \xrightarrow{\Spinmap} \OP_Q^{\oplus 4} \to \Spin(1) \to 0.
  \end{equation}
  \item[$(2)$] The dimensions of the cohomology groups of the twisted spinor bundle $\Spin$ are as follows:
 \[
   h^i(Q, \Spin(n)) = \left\{ \begin{aligned}   \frac{2n(n+1)(n+2)}{3}, && &\text{ if }  i=0 \text{ and }n\geqslant 0;\\ 
                                               -\frac{2n(n+1)(n+2)}{3},  && &\text{ if }  i=3 \text{ and }n\leqslant -2;\\ 
                                                0, && &\text{ otherwise}.
                               \end{aligned}
                       \right.
 \]
 \item[$(3)$] The dimensions of the cohomology groups of the twisted bundle $\Spin\otimes\Spin$ are as follows:
  \[
   h^i(Q, \Spin\otimes\Spin(n)) = 
   \left\{ \begin{aligned}   \frac{(2n^2 + 2n -3)(2n+1)}{3}, && &\text{ if }  i=0 \text{ and }n\geqslant 1;\\ 
                              1,  && &\text{ if }  i=1 \text{ and }n=0;\\ 
                              1,  && &\text{ if }  i=2 \text{ and }n=-1;\\ 
                             -\frac{(2n^2 + 2n -3)(2n+1)}{3},  && &\text{ if }  i=3 \text{ and }n\leqslant -2;\\ 
                              0, && &\text{ otherwise}.
           \end{aligned}
  \right.
 \]
 \item[$(4)$] If $\E$ is a vector bundle on $Q$ such that $H^1(Q,\E(n)) = H^2(Q,\E(n)) = 0$ for all $n\in \ZZ$, then
 \[
  \E = \bigoplus_{i = a}^b \Spin(i)^{\oplus m_i} \oplus \bigoplus_{j=c}^d \OP_Q(j)^{\oplus k_j},
 \]
  where $a\leqslant b$ and $c\leqslant d$ are integers, and $m_i$ and $k_j$ are non-negative integers for all $a\leqslant i\leqslant b$ and~\mbox{$c\leqslant j\leqslant d$}.
  \item[$(5)$] Let $\E$ be a vector bundle on $Q$ such that $\E = \bigoplus_{i = a}^b  \Spin(i)^{\oplus m_i} \oplus \bigoplus_{j=c}^d \OP_Q(j)^{\oplus k_j}$ for $a\leqslant b$, $c\leqslant d$ and some non-negative integers~$m_i$ and $k_j$. Then for all integers $n\in \ZZ$ one has 
  \[
  h^1(Q, \E\otimes\Spin(n+1))  = h^2(Q, \E\otimes\Spin(n))
  \]
 \end{enumerate}
\end{lemma}
\begin{proof}
 The assertion \textup{(1)} is proved in \cite[Theorem 2.8]{Ottaviani-spinors}. The exact sequence from the assertion \textup{(1)} implies assertions \textup{(2)} and \textup{(3)}. The assertion \textup{(4)} is a generalization of the Horrocks criterion to the case quadrics proved in \cite[Theorem 3.5]{Ottaviani-Horrocks_criterion}. The assertion \textup{(5)} follows from the long exact sequence of cohomology groups of the exact sequence from the assertion \textup{(1)} tensored by~$\E(n)$. 
\end{proof}
Lemma \ref{lemma: properties of Spin}(3) implies in particular that $H^1(Q,\Spin\otimes\Spin)\cong \CC$. 
By the next lemma if the tensor product of a sheaf $\F$ on $Q$ with $\Spin$ admits a non-zero first cohomology group then sometimes one can construct a map from $\Spin$ to $\F$ such that after twisting by $\Spin$ this map induces an embedding of the first cohomology groups.
\begin{lemma}\label{lemma: map from spin}
 Let~$\F$ be a coherent sheaf on a smooth quadric threefold $Q\subset \p^4$ such that~\mbox{$H^1(Q, \F) = 0$}. Fix a one-dimensional subspace $U\subset H^1(Q, \Spin\otimes \F)$. Then there exists a unique map $\varphi\colon \Spin\to \F$ such that the induced map
 \begin{equation*}
 H^1(\mathrm{Id}_{\Spin}\otimes\varphi)\colon H^1(Q, \Spin\otimes \Spin) \to H^1(Q,\Spin\otimes\F)
 \end{equation*}
 is an embedding and the image of $H^1(\mathrm{Id}_{\Spin}\otimes\varphi)$ coincides with $U$.
\end{lemma}
\begin{proof}
 Consider the short exact sequence from Lemma \ref{lemma: properties of Spin}\textup{(1)} and tensor it by $\F$. We get a new short exact sequence on $Q$. Consider its the long exact sequence of cohomology groups:
 \begin{equation*}
  \cdots\to H^0(Q,\Spin(1)\otimes \F)\xrightarrow{\delta} H^1(Q, \Spin\otimes \F) \to H^1(Q,\F)^{\oplus 4} \to \cdots
 \end{equation*}
 By assumption the group $H^1(Q,\F)$ vanishes; thus, the connecting homomorphism $\delta$ is surjective. Choose a non-zero element $\varphi$ in $\delta^{-1}(U)$. Since $H^0(Q,\Spin(1)\otimes \F)\cong \Hom_Q(\Spin, \F)$ this element defines a map. By construction this map has the necessary property.
\end{proof}

Now we prove several useful and easy assertions we will need further.
The next assertion describes the support of the cokernel of a non-zero map from the spinor bundle to $\OP_Q$. We will use this fact in the proof of Lemma \ref{lemma: map from o to spin}.

\begin{lemma}\label{lemma: map from spin to o}
 Let $\psi\colon \Spin\to \OP_Q$ be a non-zero  map of sheaves of a smooth quadric threefold $Q\subset \p^4$. Then the support of the sheaf $\Coker(\psi)$ is a line on $Q$.
\end{lemma}
\begin{proof}
 Denote the cokernel of $\psi$ by $\F$ and consider the following exact sequence:
  \[
   \Spin\xrightarrow{\psi} \OP_Q\to \F\to 0.
  \]
  Denote by $\iota\colon H\to Q$ the embbeding of a hyperplane section to $Q$. Since $\iota^*$ is a right exact functor, then the following sequence is also exact: 
  \[
   \iota^*\Spin \xrightarrow{\psi|_H} \OP_H\to \iota^*\F\to 0.
  \]
  By \cite[Corollary 2.2]{Ottaviani-spinors} one has $\iota^*\Spin$ equals $\OP_{\p^1\times \p^1}(-1,-1)$. Since the rank of $\psi|_H$ equals 1 then the support of $\iota^*\F$ is a point. Therefore, the support of $\F$ is a curve in $\p^4$ whose intersection with a hyperplane consists of one point. Therefore, the support of $\F$ is a line.
\end{proof}

The next assertion describes the support of the cokernel of a map from the trivial bundle to the dual of the  spinor bundle.

\begin{lemma}\label{lemma: map from o to spin}
 Let $\varphi\colon \OP_Q^{\oplus k}\to \Spin(1)$ be a map of sheaves of a smooth quadric threefold $Q\subset \p^4$. If the rank of~$\varphi$ equals $2$ then $\C = \Coker(\varphi)$ is a sheaf whose support is either a line on $Q$ or a hyperplane section of $Q$ or empty.
\end{lemma}
\begin{proof}
  First, assume that $k = 2$. Since the rank of $\varphi$ equals 2 then $\varphi\colon \OP_Q^{\oplus 2}\to \Spin(1)$ is an embedding. Then the cokernel $\C$ of $\varphi$ is the zero locus of the map $\det(\varphi)$. By Lemma \ref{lemma: properties of Spin}(1) we deduce that $\deg(\Spin(1)) = 1$. Thus, the supprot of $\C$ is a hyperplane section of $Q$.
  
  In order to continue the proof note that a map $\OP_Q\to \Spin(1)$ is induced by a global section $H^0(Q,\Spin(1))\cong \CC^4$. By \cite[Theorem 2.8]{Ottaviani-spinors} the map $\Phi\colon\OP_Q^{\oplus 4}\to \Spin(1)$ in the exact sequence \eqref{eq: spin} is induced by the whole space of global sections $H^0(Q,\Spin(1))$. Thus, our map $\varphi$ factors through this map, i.e. there exists a map $\alpha\colon \OP_Q^k \to \OP_Q^4$ such that the following diagram is commutative:
\[\begin{tikzcd}[ampersand replacement=\&]
	\&\& {\OP_Q^{\oplus k}} \& {\Spin(1)} \& \C \& 0 \\
	0 \& \Spin \& {\OP_Q^{\oplus 4}} \& {\Spin(1)} \& 0
	\arrow["\varphi", from=1-3, to=1-4]
	\arrow["\alpha"', from=1-3, to=2-3]
	\arrow[from=1-4, to=1-5]
	\arrow[Rightarrow, no head, from=1-4, to=2-4]
	\arrow[from=1-5, to=1-6]
	\arrow[from=2-1, to=2-2]
	\arrow[from=2-2, to=2-3]
	\arrow["\Phi", from=2-3, to=2-4]
	\arrow[from=2-4, to=2-5]
\end{tikzcd}\]
Thus, if $k\geqslant 4$ and the map $\alpha$ is surjective, then $\C$ equals 0. Now assume that $k$ equals $3$ and $\alpha$ is an embedding. Consider the map of the following two exact sequences:
\[\begin{tikzcd}[ampersand replacement=\&]
	0 \& {\OP_Q^{\oplus 3}} \& {\OP_Q^{\oplus 4}} \& {\OP_Q} \& 0 \\
	0 \& {\Spin(1)} \& {\Spin(1)} \& 0 \& {}
	\arrow[from=1-1, to=1-2]
	\arrow["\alpha", from=1-2, to=1-3]
	\arrow["\varphi"', from=1-2, to=2-2]
	\arrow[from=1-3, to=1-4]
	\arrow["\Phi", from=1-3, to=2-3]
	\arrow[from=1-4, to=1-5]
	\arrow[from=1-4, to=2-4]
	\arrow[from=2-1, to=2-2]
	\arrow[Rightarrow, no head, from=2-2, to=2-3]
	\arrow[from=2-3, to=2-4]
\end{tikzcd}\]
 By the snake lemma we get the following exact sequence of sheaves on the quadric threefold $Q$:
 \[
  \Spin\xrightarrow{\psi} \OP \to \C\to 0.
 \]
 The support of the cokernel $\C = \Coker(\varphi)$ cannot coinside with $Q$. Thus, the map $\psi$ is non-zero. By Lemma~\ref{lemma: map from spin to o} this implies that the support of $\C$ is a line.
 
 It remains to prove the assertion for the case when $k = 3$ but $\alpha$ is not an embedding and $k = 4$ but $\alpha$ is not surjective. However, one can see that the map $\varphi$ in these cases factors either through $\OP_Q^{\oplus 2}\to \Spin(1)$ or through $ \OP_Q^{\oplus 3}\to \Spin(1)$ where $\alpha'\colon  \OP_Q^{\oplus 3} \to \OP_Q\otimes H^0(Q,\Spin(1))$ is an embedding and this implies the result.
\end{proof}

\subsection{Geometry of a nodal intersection of a quadric and quartic threefolds}\label{subsection: even sets of nodes}
Consider a smooth quadric hypersurface $Q$ and a quartic hypersurface $V_4$ in $\p^4$ such that the intersection~\mbox{$B = Q\cap V_4$} is an irreducible reduced nodal surface. Denote by $\iota\colon B\to Q$ the embedding of the surface $B$ into $Q$ and by~\mbox{$\sigma\colon \B\to B $} the blow-up of $B$ in $\Sing(B)$; it is a crepant resolution of singularities of $B$. 
 \begin{lemma}\label{lemma: Pic(B) is torsion-free}
  Let $B$ be a nodal complete intersection of a smooth quadric hypersurface $Q$ and a quartic hypersurface in $\p^4$ and let $\B$ be the blow-up of $B$ in $\Sing(B)$. Then the Picard group $\Pic(\B)$ is torsion-free.
 \end{lemma}
 \begin{proof}
  By \cite[Theorem 1]{Tyurina} one can construct a smooth $1$-dimensional family $\pi\colon \mathscr{B}\to T$ such that the fiber over the point $0\in T$ is isomorphic to $\B$ and a fiber $\mathscr{B}_t$ is a smooth complete intersection of a quadric and quartic hypersurfaces in $\p^4$ for all $t\in T\setminus \{0\}$. Thus, the surface $\B$ is deformation equivalent to $\mathscr{B}_t$ where $t\ne 0$.
  In particular, this implies that $\B$ is diffeomorphic to $\mathscr{B}_t$. 
  
  By the universal coefficient theorem the torsion subgroup of $H^2(\mathscr{B}_t, \ZZ)$ equals the torsion subgroup of the homology group $H_1(\mathscr{B}_t, \ZZ)$ which vanishes by the Lefschetz hyperplane theorem. Thus, the group $H^2(\mathscr{B}_t, \mathbb{Z})$ is torsion-free and so is the group $H^2(\B, \mathbb{Z})$. Finally, since $H^1(\B,\OP_{\B}) = 0$ the exponential exact sequence implies the result.
 \end{proof}
 
 Denote by $H$ the inverse image of the divisor class of a hyperplane section of $\p^4$ to the surface $\B$.
 For a node~$p$ on the surface $B$ denote by $E_p$ the exceptional divisor of $\sigma\colon\B\to B$ over this point. Then $H$ and $E_p$ for all nodes $p\in \Sing(B)$ are the classes of Cartier divisors on the surface $\widetilde{B}$.
 \begin{defin}\label{definition: even sets of nodes}
 We say that the set $w$ of nodes on the nodal surface $B$ is a \emph{$\frac{\delta}{2}$-even set of nodes on $B$} for~$\delta = 0$ or $1$ if 
 \begin{equation*}
  E_w = \delta H +\sum_{p\in w} E_p \in 2\Pic(\B).
 \end{equation*}
 A $\frac{\delta}{2}$-even set of nodes $w$ on $B$ is called {\itshape minimal} if $w$ does not contain a proper non-empty subset $w'$ which is a $\frac{\delta'}{2}$-even set of nodes on $B$ for $\delta' = 0$ or $1$.
 \end{defin}
 Choose the class $\frac{E_w}{2}$ in $\Pic(\B)$. If $\Pic(\widetilde{B})$ has 2-torsion, there are more than one choice for such class. However, by Lemma \ref{lemma: Pic(B) is torsion-free} the Picard group $\Pic(\widetilde{B})$ is torsion-free; thus, the class $\frac{E_w}{2}$ is unique. Then one can define the following reflexive sheaf on the surface $B$:
 \begin{equation*}\label{eq: definition of F}
  \F_B = \pi_*\OP_{\B}\left(-\frac{E_w}{2}\right).
 \end{equation*}
 We call $\F_B$ the {\itshape sheaf associated to the $\frac{\delta}{2}$-even set of nodes $w$ on the surface $B$}. Denote by $\F$ the direct image $\F = \iota_*\F_B$, it is the {\itshape sheaf associates to $w$ on the threefold $Q$}. These sheaves encodes many properties of the surface $B$ and the set of nodes $w$, see \cite{Casnati_Catanese}. By construction the rank of the sheaf $\F_B$ equals~$1$ at all points of $B$ except $w$ where it equals $2$. 
 
  Fix a $\frac{\delta}{2}$-even set of nodes $w$ on a nodal complete intersection of a quadric and quartic threefolds in $\p^4$. Then the sheaves $\F_B$ on $B$ and $\F$ on $Q$ associated to $w$ has the following properties.
 \begin{lemma}\label{lemma: properties of F}
 Let $B$ be a nodal complete intersection of a smooth quadric hypersurface $Q$ and a quartic hypersurface in $\p^4$, let $w$ be a $\frac{\delta}{2}$-even set of nodes on $B$ and let $\F_B$ and $\F$ be the sheaves associated to $w$ on $B$ and $Q$ respectively. Then the following assertions hold:
  \begin{enumerate}
   \item[$(1)$] The supports of sheaves $\F_B$ and $\F$ equal $B$ and the projective dimension of $\F$ on the quadric threefold $Q$ is~$1$. There is a non-degenerate pairing $S^2\F_B\to \OP_B(-\delta)$.
   \item[$(2)$] $\ExtSh_{Q}^1(\F,\OP_Q) \cong \F(4+\delta)$.
   \item[$(3)$] $h^i(B,\F_B(n)) = h^{2-i}(B, \F_B(\delta+1 - n))$ and $h^i(Q,\F(n)) = h^{2-i}(Q, \F(\delta+1 - n))$ for all $0\leqslant i\leqslant 2$ and all $n \in\ZZ$.
   \item[$(4)$] $h^i(B,\F_B\otimes\iota^*\Spin(n)) = h^{2-i}(B, \F_B\otimes\iota^*(\Spin)(\delta+2 - n))$ and $h^i(Q,\F\otimes\Spin(n)) = h^{2-i}(Q, \F\otimes\Spin(\delta+2 - n))$ for all $0\leqslant i\leqslant 2$ and all $n \in\ZZ$.
   \item[$(5)$] $h^1(B,\F_B(-n)) =h^1(Q,\F(-n)) = 0$ for all $n>0$.
   \item[$(6)$] If $w$ is a minimal $\frac{\delta}{2}$-even sets of nodes, then $h^1(B,\F_B) = h^1(Q,\F) = 0$.
   \item[$(7)$] $\chi(\F_B(n)) = \chi(\F(n)) = (2n - \delta)(2n - \delta - 2) - \frac{|w|}{4} + 6$.
   \item[$(8)$] $\chi(\F_B\otimes \iota^*(\Spin)(n)) =\chi(\F\otimes \Spin(n)) = 8n(n-2-\delta) + 16+ 10 \delta - \frac{|w|}{2}$.
   \item[$(9)$] If $|w|\geqslant 32$ then $w$ is not minimal.
  \end{enumerate}
 \end{lemma}
 \begin{proof}
  The assertion \textup{(1)} is analogous to \cite[Lemma 1.1]{Casnati_Catanese}. The existence of the non-degenerate pairing follows from \cite[Proposition 3.1]{Casnati_Catanese}. 
  Consider the following exact sequence of sheaves on the smooth surface~$\B$:
  \[
   0\to \OP_{\B}\left(\frac{\delta H - \sum_{p\in w}E_p}{2}\right)\to\OP_{\B}\left(\frac{\delta H + \sum_{p\in w}E_p}{2}\right)\to \bigoplus_{p\in w} \OP_{E_p}\left(\frac{\delta H + \sum_{p\in w}E_p}{2}\right) \to 0
  \]
  All the cohomology groups of the sheaf $\OP_{E_p}\left(\frac{\delta H + \sum_{p\in w}E_p}{2}\right) = \OP_{E_p}(-1)$ supported on the curve $E_p\cong\p^1$ vanish. Thus, taking the direct image to $B$ of this exact sequence we get the following isomorphism of sheaves: 
  \begin{equation}\label{eq: duality for F}
   \F_B(\delta)\cong \HomSh_B(\F_B,\OP_B).
  \end{equation}
  Then the Grothendieck duality in the form \cite[Theorem III.11.1]{Hartshorne_Grothendieck_duality} applied for the embedding $\iota\colon B\to Q$ implies the second equality of the assertion \textup{(2)}. 
  
  Fix a vector bundle $\E$ on $Q$. Then by projection formula one has $\F\otimes \E \cong \iota_*(\F_B\otimes \iota^*\E)$. 
  Thus, since the closed embedding is an affine morphism one has $h^i(B,\F_B\otimes \iota^*( \E))   = h^i(Q,\F\otimes \E)$ for all $0\leqslant i\leqslant 2$. Therefore, it suffices to prove the equality in the assertions \textup{(3-8)} only for the sheaf $\F_B$.
  
  Now the assertions \textup{(3)} and \textup{(4)} follows from the isomorphism \eqref{eq: duality for F} and the Serre duality on the surface~$B$.  
  The proof of the assertion \textup{(5)} is analogous to one proved in \cite[Section 3]{Casnati_Catanese}. The assertion \textup{(6)} follows from \cite[Lemma 2]{Beauville}. In order to prove next assertions note that the Todd genus of the surface $\B$ and the Chern characters of vector bundles~\mbox{$ \sigma^*(\iota^*(\Spin))$, $ \OP_{\B}(nH)$} and~$\OP_{\widetilde{B}}(-E_w/2)$ are as follows:
  \begin{align*}
   &\mathrm{td}(\B) = 1 - \frac{1}{2}H + \frac{3}{4}H^2;\\
   &\mathrm{ch}(\sigma^*(\iota^*(\Spin))) = 2 - H;\\
   &\mathrm{ch}(\OP_{\B}(nH)) = 1+ nH + \frac{n^2}{2}H^2;\\
   &\mathrm{ch}\left(\OP_{\widetilde{B}}(-E_w/2)\right) = 1 - \frac{1}{2}\left(\delta H+ \sum_{p\in w}E_p\right) + \frac{1}{8}\left(\delta H^2+ \sum_{p\in w}E_p^2\right).
  \end{align*}
  Using the Hirzebruch--Riemann--Roch formula we prove assertions \textup{(7)} and \textup{(8)}. The last assertion is analogous to \cite[Section 2]{Beauville}.
 \end{proof}
 Lemma \ref{lemma: properties of F} implies the following properties of a $\frac{\delta}{2}$-even sets of nodes on a complete intersection of a smooth quadric threefold and a quartic threefold in $\p^4$.
 \begin{corollary}\label{corollary: properties of w}
  Let $B$ be a nodal complete intersection of a smooth quadric hypersurface $Q$ and a quartic hypersurface in $\p^4$ and let $w$ be a minimal $\frac{\delta}{2}$-even set of nodes. Then $|w|\leqslant 28$ and it is divisible by $4$. Moreover, if~\mbox{$\delta = 1$} then $|w|\equiv 4 \mod 8$.
 \end{corollary}
 \begin{proof}
  The result follows from assertions (7) and (9) of Lemma \ref{lemma: properties of F}.
 \end{proof}

 We say that the sheaf $\F$ associated to a $\frac{\delta}{2}$-even set of nodes on $Q$ fits into a {\itshape Casnati--Catanese resolution} if there exists a vector bundle~$\E$ on $Q$, a map $\varphi\in\Hom_Q(\E,\F)$ and a map $\Phi\in \Hom_Q(\E^{\vee}(-4-\delta),\E)$ such that the following sequence is exact:
 \begin{equation}\label{eq: main exact sequence}
  0\to \E^{\vee}(-4-\delta)\xrightarrow{\Phi}\E\xrightarrow{\varphi} \F \to 0.
 \end{equation}
 We say that the Casnati--Catanese resolution of the sheaf $\F$ is {\itshape symmetric} if the map $\Phi$ lies in the subspace~\mbox{$H^0(Q,(S^2\E)(4+\delta)) \subset \Hom_Q(\E^{\vee}(-4-\delta),\E)$}. 
 \begin{lemma}[{\cite[Section 2]{Casnati_Catanese}}]\label{lemma: CC plus H1 vanishes implies symmetricity}
  Let $B$ be a nodal complete intersection of a smooth quadric hypersurface~$Q$ and a quartic hypersurface in $\p^4$ and let $w$ be a minimal $\frac{\delta}{2}$-even set of nodes. If the sheaf $\F$ associated with $w$ on $Q$ fits a Casnati--Catanese resolution \eqref{eq: main exact sequence} and $H^1(Q,(S^2\E^{\vee})(-4-\delta)) = 0$ then $\F$ fits a symmetric Casnati--Catanese resolution with the same bundle $\E$.
 \end{lemma}
 \begin{proof}
   Here we repeat the argument in \cite[Section 2]{Casnati_Catanese}.
   Consider the exact sequence \eqref{eq: main exact sequence} and apply the functor $\HomSh_Q(-,\OP_Q(-4-\delta))$ we get the following new exact sequence:
   \[
    0\to \E^{\vee}(-4-\delta)\xrightarrow{\Phi^{\vee}}\E\to \ExtSh_Q^1(\F,\OP_Q(-4-\delta))\to 0.
   \]
   By Lemma \ref{lemma: properties of F}(2) there exists an isomorphism $s_0\colon \F\to \ExtSh^1_Q(\F,\OP_Q(-4-\delta))$. Our goal is to show that there exist an isomorphism $s_1\colon \E\to \E$ such that the following diagram is commutative:
  \begin{equation}\label{eq: diagram map from CC to dual CC}
    \begin{tikzcd}[ampersand replacement=\&]
	0 \& {\E^{\vee}(-4-\delta)} \& \E \& \F \& 0 \\
	0 \& {\E^{\vee}(-4-\delta)} \& \E \& {\ExtSh_Q(\F,\OP_Q(-4-\delta))} \& 0
	\arrow[from=1-4, to=1-5]
	\arrow[from=2-4, to=2-5]
	\arrow["\varphi", from=1-3, to=1-4]
	\arrow[from=2-3, to=2-4]
	\arrow["{\Phi^{\vee}}", from=2-2, to=2-3]
	\arrow["\Phi", from=1-2, to=1-3]
	\arrow["{s_1^{\vee}}", from=1-2, to=2-2]
	\arrow["{s_1}", from=1-3, to=2-3]
	\arrow["{s_0}", from=1-4, to=2-4]
	\arrow[from=1-1, to=1-2]
	\arrow[from=2-1, to=2-2]
   \end{tikzcd}
  \end{equation}
  Once the isomorphism $s_1$ is constructed we replace $\Phi$ be the composition $s_1\circ\Phi$. By commutativity of the diagram \eqref{eq: diagram map from CC to dual CC} we have $s_1\circ \Phi = \Phi^{\vee}\circ s_1^{\vee}$. Therefore, $s_1\circ\Phi$ lies in $H^0(Q,(S^2\E)(4+\delta))$ and the proof is complete.
  Using the exact sequence \eqref{eq: main exact sequence} one can construct the following resolution of the sheaf $S^2\F$:
  \[
   0\to \Lambda^2\E^{\vee}(-8-2\delta)\to \E^{\vee}(-4-\delta)\otimes \E \to S^2\E\to S^2\F\to 0.
  \]
  By Lemma \ref{lemma: properties of F}(1) there exists a map $\varphi_0\colon S^2\F \to \OP_{Q}(-\delta)$. We are going to show that this map extends to the following map between exact sequences:
\begin{equation}\label{eq: diagram with symm square of CC}
\begin{tikzcd}[ampersand replacement=\&]
	0 \& {\Lambda^2\E^{\vee}(-8-2\delta)} \& {\E^{\vee}(-4-\delta)\otimes \E} \& {S^2\E} \& {S^2\F} \& 0 \\
	\& 0 \& {\OP_Q(-4-\delta)} \& {\OP_Q(-\delta)} \& {\OP_B(-\delta)} \& 0
	\arrow[from=1-1, to=1-2]
	\arrow[from=1-5, to=1-6]
	\arrow[from=2-5, to=2-6]
	\arrow["\alpha", from=1-4, to=1-5]
	\arrow[from=2-4, to=2-5]
	\arrow[from=2-3, to=2-4]
	\arrow["\beta", from=1-3, to=1-4]
	\arrow[from=1-2, to=1-3]
	\arrow[from=2-2, to=2-3]
	\arrow["0"', from=1-2, to=2-2]
	\arrow["{\varphi_2}", from=1-3, to=2-3]
	\arrow["{\varphi_1}", from=1-4, to=2-4]
	\arrow["{\varphi_0}", from=1-5, to=2-5]
\end{tikzcd}
\end{equation}
 The lower exact sequence here is just a twist of the standard exact sequence for the sheaf $\OP_B$ in $Q$. Denote by $\alpha\colon S^2\E\to S^2\F$ and $\beta\colon \E^{\vee}(-4-\delta)\otimes \E\to S^2\E$ the maps in the upper exact sequence. Apply the functor $\Hom_Q(S^2\E,-)$ to the lower exact sequence in the diagram. We get the following long exact sequence:
 \[
  \dots\to \Hom_Q(S^2\E,\OP_Q(-\delta))\to \Hom_Q(S^2\E,\OP_B(-\delta)) \to \Ext^1(S^2\E, \OP_Q(-4-\delta))
 \]
 Since by our assumption $\Ext^1(S^2\E, \OP_Q(-4-\delta))\cong H^1(Q,(S^2\E^{\vee})(-4-\delta))=0$ the composition $\varphi_0\circ\alpha\colon S^2\E\to \OP_B(-\delta)$ can be lifted to the map $\varphi_1\colon S^2\E\to \OP_Q(-\delta)$. If we set $\varphi_2 = \frac{1}{f}\varphi_1\circ\beta$ where $f$ is the equation of the divisor $B$ inside $Q$, then we gets the necessary maps such that the diagram \eqref{eq: diagram with symm square of CC} commutes. 
 
 Finally, the map $\varphi_2 \in \Hom_Q(\E^{\vee}(-4-\delta)\otimes \E, \OP_Q(-4-\delta)) \cong \Hom_Q(\E,\E)$ defines the map $s_1\colon \E \to \E$ such that the diagram \eqref{eq: diagram map from CC to dual CC} also commutes. This finishes the proof.
 \end{proof}

 The next assertion describes how look intersections of $\frac{\delta}{2}$-even sets of nodes on the surface $B$ with lines and conics in the projective space $\p^4$.
 \begin{lemma}\label{lemma: lines and conics on B containing 1/2-even sets}
  Let $B$ be a nodal complete intersection of a smooth quadric hypersurface $Q$ and a quartic hypersurface in $\p^4$ and let $w$ be a $\frac{1}{2}$-even set of nodes on $B$. Then the following assertions hold:
  \begin{enumerate}
   \item[\textup{(1)}] If $L\subset B$ is a line in $\p^4$ then $|L\cap w|$ is odd. There is no line $L\subset B$ such that $L$ contains $w$.
   \item[\textup{(2)}] If $L$ is a line in $\p^4$ then $|L\cap w|\leqslant 4$.
   \item[\textup{(3)}] If $C$ is a conic in $\p^4$ then $|C\cap w|\leqslant 8$.
  \end{enumerate}
 \end{lemma}
 \begin{proof}
  Denote by $\widetilde{L}$ the proper preimage of $L$ to $\widetilde{B}$. 
  Then $\widetilde{L}$ is an effective divisor on $\widetilde{B}$ with the following intersection properties:
  \begin{align*}
   \widetilde{L}\cdot H = 1; &&\widetilde{L}\cdot E_{p} = 1 \text{ for all } p\in L\cap w; &&\widetilde{L}\cdot E_{p} = 0 \text{ for all } p\in w\setminus L.
  \end{align*}
  The divisor class $E_w = \delta H - \sum_{p\in w} E_p$ lies in the group $2\Pic(\B)$. Thus, the product $L\cdot E_w = 1 + |L\cap w|$ is divisible by $2$. Therefore, Corollary \ref{corollary: properties of w} implies the assertion \textup{(1)}.

  Let $(z_0:z_1 :z_2 :z_3 :z_4)$ be homogeneous coordinates on the projective space $\p^4$ and let $q(z_0,z_1,z_2,z_3,z_4)$ and $f(z_0,z_1,z_2,z_3,z_4)$ be homogeneous polynomials of degrees $2$ and $4$ respectively such that $Q = \{q = 0\}$ and $B = \{q = f = 0\}$. Denote by $M$ the Jacobian matrix of $B$:
\[
 M = \begin{pmatrix}
      \frac{\partial q}{\partial z_0} & \frac{\partial q}{\partial z_1} & \frac{\partial q}{\partial z_2} &\frac{\partial q}{\partial z_3} & \frac{\partial q}{\partial z_4}    \\
      \frac{\partial f}{\partial z_0} & \frac{\partial f}{\partial z_1} & \frac{\partial f}{\partial z_2} &\frac{\partial f}{\partial z_3} & \frac{\partial f}{\partial z_4}    
     \end{pmatrix}.
\]
Then the singular locus of $B$ is as follows $\Sing(B) = \left\{ z\in\p^4\ | \ q(z) = f(z) = 0,\ \rk(M(z)) = 1\right\}$. Since determinants of any $2\times 2$ minor of the matrix $M$ is a polynomial of degree $4$ we get that $\Sing(B)$ is an intersection of several quartic hypersurfaces in $\p^4$. 

Since singularities of $B$ are isolated the intersection of $\Sing(B)$ with a curve is a finite set. Thus, if $L$ is a line then the cardinality of $\Sing(B)\cap L$ is at most $4$. If $C$ is a conic, then the cardinality of $\Sing(B)\cap C$ is at most $8$. Thus, assertions \textup{(2)} and \textup{(3)} hold.
 \end{proof}

 \begin{lemma}[{see \cite[Proposition 4.2]{Jaffe_Ruberman}}]\label{lemma: JR restriction on cohomology groups of F}
 Let $B$ be a nodal complete intersection of a smooth quadric hypersurface and a quartic hypersurface in $\p^4$ and let $w$ be a $\frac{\delta}{2}$-even set of nodes on~$B$. If $\F$ is the sheaf associated to $w$ on $Q$ then  $h^0(Q, \F(1)) \leqslant 2$.
 \end{lemma}
 \begin{proof} 
 Let $\F_B$ be the sheaf associated to $w$ on $B$. Then $\F = \iota_*\F_B$ and $h^0(Q, \F(1)) = h^0(B,\F_B(1))$ since the closed embedding $\iota\colon B\to Q$ is an affine morphism.
 
 Assume that $h^0(B, \F_B(1))\geqslant 3$ then by projection formula we have also $h^0\left(\B,\frac{H - \sum_{p\in w}E_p}{2}\right)\geqslant 3$. This implies that $h^0(\B, H - \sum_{p\in w}E_p)\geqslant 3$; i.e. there exists three linearly independent hyperplane sections $H_1, H_2$ and $H_3$ in the surface $B$ which contain the $\frac{\delta}{2}$-even set of node $w$. Therefore, $w$ lies on a line in $\p^4$ and this contradicts Lemma \ref{lemma: lines and conics on B containing 1/2-even sets}.
 \end{proof}

 \subsection{Cokernel of a map between vector bundles}
 Let $X$ be a smooth variety, let $\LA$ and $\LB$ be vector bundles of equal ranks on~$X$ and let $\Phi$ be an embedding of $\OP_X$-sheaves $\Phi\colon \LA\to \LB$. Let $B$ be the zero locus of the determinant map $B = \{\det(\Phi) = 0\}$. By construction the support of the cokernel $\Coker(\Phi)$ coincides with the reduced scheme $B_{\mathrm{red}}$. We study the influence of the geometric properties of the variety $B$ on the map~$\Phi$.
 
 In view of the following assertion if $\Phi$ maps a direct summand of $\LA$ isomorphically to a direct summand of $\LB$ then $B$ can be realized as a support of the cokernel of another embedding.
 \begin{lemma}\label{lemma: excluding the isomorphism}
 Let $\LA_1$, $\LA_2$, $\LB_1$ and $\LB_2$ be vector bundles on a smooth variety $X$ and let $\Phi_{ij}\in \Hom_X(\LA_i,\LB_j)$ be maps for all $i = 1,2$ and $j  = 1,2$ such that $\Phi = (\Phi_{ij})\colon \LA_1\oplus \LA_2 \to \LB_1\oplus \LB_2$ is an embedding of sheaves. 
  
 If $\rk(\LA_1)+\rk(\LA_2) = \rk(\LB_1)+\rk(\LB_2)$ and $\Phi_{11}$ is an isomorphism then there exists $\Phi_{22}'\in \Hom(\LA_2,\LB_2)$ such that $\{\det(\Phi) = 0\}_{\mathrm{red}}= \{\det(\Phi_{22}') = 0\}_{\mathrm{red}}$.
  \end{lemma}
  \begin{proof}
   Let $i_{\LA_1}\colon \LA_1 \to \LA_1\oplus \LA_2$ and $i_{\LB_1}\colon \LB_1 \to \LB_1\oplus \LB_2$ be the embeddings of the first summands into the direct sum. 
   Let $\pi_{\LA_j}\colon \LA_1\oplus \LA_2\to \LA_j$ and  $\pi_{\LB_j}\colon \LB_1\oplus \LB_2\to \LB_j$ be the projections from the direct sum to the $j$-th summand. Denote by $\alpha$ the following map:
   \begin{align*}
   \alpha\colon \LB_1 \to \LB_1\oplus \LB_2;&& \alpha(\nu) = \left(\nu, \Phi_{12}\left(\Phi_{11}^{-1}(\nu)\right)\right).
   \end{align*}
   Here $\nu$ is a local section of $\LB_1$.
   Then the following diagram is commutative:
   \begin{equation}\label{eq_es_excluding_the_isomorphism}
   \begin{tikzcd}[ampersand replacement=\&]
	{\LA_1} \& {\LA_1\oplus\LA_2} \\
	{\LB_1} \& {\LB_1\oplus\LB_2}
	\arrow["{i_{\LA_1}}", from=1-1, to=1-2]
	\arrow["\alpha", from=2-1, to=2-2]
	\arrow["{\Phi_{11}}"', from=1-1, to=2-1]
	\arrow["\Phi", from=1-2, to=2-2]
\end{tikzcd}
\end{equation}
This implies that $\Phi$ induces a map between the cokernels of $i_{\LA_1}$ and $\alpha$. Note that $\Coker(i_{\LA_1}) \cong \LA_2$. Moreover, one has $\Coker(\alpha) \cong \LB_2$ since the following sequence is exact:
\[\begin{tikzcd}[ampersand replacement=\&]
	0 \& {\LB_1} \& {\LB_1\oplus\LB_2} \& {\LB_2} \& 0
	\arrow[from=1-1, to=1-2]
	\arrow["\alpha", from=1-2, to=1-3]
	\arrow["\beta", from=1-3, to=1-4]
	\arrow[from=1-4, to=1-5]
\end{tikzcd}\]
Here $\beta\colon \LB_1\oplus\LB_2\to \LB_2$ is the map $\beta(\nu_1,\nu_2) =  \nu_2 -  \Phi_{12}\left(\Phi_{11}^{-1}(\nu_1)\right)$ for a local section  $(\nu_1,\nu_2)$ of the vector bundle $\LB_1\oplus \LB_2$.
Therefore, the commutative diagram \eqref{eq_es_excluding_the_isomorphism} induces the embedding of sheaves $\Phi'_{22}\colon \LA_2 \to \LB_2$ such that the following diagram commutes:
\[\begin{tikzcd}[ampersand replacement=\&]
	0 \& {\LA_1} \& {\LA_1\oplus\LA_2} \& {\LA_2} \& 0 \\
	0 \& {\LB_1} \& {\LB_1\oplus\LB_2} \& {\LB_2} \& 0
	\arrow[from=2-1, to=2-2]
	\arrow["\alpha", from=2-2, to=2-3]
	\arrow["\beta", from=2-3, to=2-4]
	\arrow[from=2-4, to=2-5]
	\arrow[from=1-1, to=1-2]
	\arrow["{i_{\LA_1}}", from=1-2, to=1-3]
	\arrow["{\pi_{\LA_2}}", from=1-3, to=1-4]
	\arrow[from=1-4, to=1-5]
	\arrow["{\Phi_{11}}"', from=1-2, to=2-2]
	\arrow["\Phi"', from=1-3, to=2-3]
	\arrow["{\Phi_{22}'}", from=1-4, to=2-4]
\end{tikzcd}\]
Since $\Phi_{11}$ is an isomorphism the snake lemma implies that $\Coker(\Phi)$ is isomorphic to $\Coker(\Phi_{22}')$. Therefore, the supports of these sheaves are isomorphic and we get the result.
 \end{proof}
 The next assertion shows that if $B$ is an irreducible scheme then there is an additional condition on the map $\Phi$.
 \begin{lemma}\label{lemma: excluding the embedding of full rank}
 Let $\LA_1$, $\LA_2$, $\LB_1$ and $\LB_2$ be vector bundles on a smooth variety $X$ and let $\Phi_{ij}\in \Hom_X(\LA_i,\LB_j)$ be maps for all $i = 1,2$ and $j  = 1,2$ such that $\Phi = (\Phi_{ij})\colon \LA_1\oplus \LA_2 \to \LB_1\oplus \LB_2$ is an embedding of sheaves. 
  
 If $\rk(\LA_1)+\rk(\LA_2) = \rk(\LB_1)+\rk(\LB_2)$ and one has $\Phi_{21}\colon \LA_1\to \LB_2$ is the zero map and $\deg(\LA_1)\ne \deg(\LB_1)$ then $\rk(\LA_1) \leqslant \rk(\LB_1)$ and in the case of the equality either $B = \{\det(\Phi) = 0\}$ coincides with $\{\det(\Phi_{11}) = 0\}$ or $B$ is reducible or non-reduced.
 \end{lemma}
 \begin{proof}
  Since the map $\Phi$ defines an embedding and $\Phi_{21}$ is zero  $\Phi_{11}$ should also be an embedding and the scheme $\{\det(\Phi) = 0\}$ coincides with the scheme $\{\det(\Phi_{11})\cdot\det(\Phi_{22})) = 0\}$. Thus, we get the inequality on ranks of bundles $\LA_1$ and $\LB_1$. 
  
  Assume now that $\rk(\LA_1) = \rk(\LB_1)$. Then the following diagram commutes:
\[\begin{tikzcd}[ampersand replacement=\&]
	0 \& {\LA_1} \& {\LA_1\oplus\LA_2} \& {\LA_2} \& 0 \\
	0 \& {\LB_1} \& {\LB_1\oplus\LB_2} \& {\LB_2} \& 0
	\arrow[from=1-1, to=1-2]
	\arrow[from=1-2, to=1-3]
	\arrow[from=1-3, to=1-4]
	\arrow[from=1-4, to=1-5]
	\arrow[from=2-1, to=2-2]
	\arrow[from=2-2, to=2-3]
	\arrow[from=2-3, to=2-4]
	\arrow[from=2-4, to=2-5]
	\arrow["{\Phi_{11}}"', hook', from=1-2, to=2-2]
	\arrow["\Phi"', hook', from=1-3, to=2-3]
	\arrow["{\Phi_{22}}", from=1-4, to=2-4]
\end{tikzcd}\]
Then by snake lemma we conclude that either $\Phi_{11}$ is an isomorphism or the cokernel of $\Phi$ is supported on the union of supports of the cokernels of $\Phi_{11}$ and $\Phi_{22}$. This implies the result.
 \end{proof}

 \subsection{Defect of a set of points} 
 We start with the definition of the invariant of a finite set of points in a smooth quadric threefold $Q\subset \p^4$. 
 \begin{defin}\label{definition: defect}
  Let $Q$ be a smooth quadric threefold  in $\p^4$ and let $w\subset Q$ be a finite set of points. Denote by  $\I_w$ the sheaf of ideals  of the set of points $w$. Then the {\it defect of $w$} is the following number:
  \begin{equation*}
   d(w) = h^0(Q, \I_w(3)) - (h^0(Q,\OP_Q(3)) - |w|) = h^0(Q, \I_w(3)) - 30 + |w|.
  \end{equation*}
 \end{defin}
 The defect shows how the dimension of the space of divisors in the linear system $|\OP_Q(3)|$ passing through the set of points $w$ differs from the expected dimension. The defect has the following useful property.
 \begin{lemma}[{\cite[Lemma 4.6]{Kuznetsova_sextic_double_solids}}]\label{lemma: defect of a subset}
 Let $w$ be a set of points on a smooth quadric threefold $Q$  and let $w'$ be its subset. Then $d(w')\leqslant d(w)$.
\end{lemma} 
 Let $B$ be a nodal complete intersection of a smooth quadric hypersurface $Q$ and a quartic hypersurface in the projective space $\p^4$. Assume that $w_1$ and~$w_2$ are a $\frac{\delta_1}{2}$-even and a $\frac{\delta_2}{2}$-even sets of nodes on $B$. Following Catanese one can define a new set of nodes $w_1+w_2$ on $B$ which we call the \emph{sum of $w_1$ and $w_2$}. Namely, we set 
 \[
 w_1 + w_2 = (w_1\cup w_2)\setminus (w_1\cap w_2). 
 \]
 By \cite{Endrass99} one can see that if $\delta = \delta_1 +\delta_2 \mod 2$, then $w_1+w_2$ is a $\frac{\delta}{2}$-even set of nodes on $B$. Denote by~$\codes$ the abelian group whose elements are~\mbox{$\frac{\delta}{2}$-ev}\-en sets of nodes on $B$ and the group law described above. Since all non-zero elements of $\codes$ are of order $2$ we see that the order of $|\codes|$ is a power of $2$. 
 
 The following assertion gives a criterion for the existence of an Artin--Mumford obstruction to rationality of a double cover of a quadric threefold in termes of the defect of $\Sing(B)$ and the order of $\codes$.
 \begin{theorem}[{\cite[Corollary 2]{Cynk_defect}}]\label{theorem: inequality between codes and defect}
  Let $B$ be a nodal complete intersection of a smooth quadric hypersurface $Q$ and a quartic hypersurface in the projective space $\p^4$ and let $\codes$ be the group of even sets of nodes on $B$. Then the double cover of $Q$ ramified over $B$ admits an Artin--Mumford obstruction to rationality if one has the following inequality:
  \begin{equation*}
   |\codes|> 2^{d(\Sing(B))},
  \end{equation*}
  where $d(\Sing(B))$ is the defect of the set $\Sing(B)$ of singular points in $B$.
 \end{theorem}
 The next assertion gives a relation between defects of $\frac{\delta}{2}$-even sets of nodes on the nodal surface $B$ with the defect of all singular points of $B$.
\begin{theorem}[{\cite[Corollary 4.10]{Kuznetsova_sextic_double_solids}}]\label{theorem: if all defects equal 1}
  Let $Q$ be a smooth quadric threefold and let $B$ be a nodal surface in $Q$. If for any non-zero $\frac{\delta}{2}$-even set of nodes $w\in \codes$ one has~$d(w)\geqslant 1$ then~$|\codes|= 2^{d(\Sing(B))}$.
\end{theorem}
Theorems \ref{theorem: inequality between codes and defect} and \ref{theorem: if all defects equal 1} lead to the following useful corollary.
\begin{corollary}\label{corollary: AM implies w with 0 defect}
 Let $Q$ be a smooth quadric threefold and let $B$ be a nodal surface in $Q$. If the double cover of $Q$ ramified over $B$ admits an Artin--Mumford obstruction to rationality then there is a non-empty~\mbox{$\frac{\delta}{2}$-ev}\-en set of nodes $w\subset B$ such that $d(w) = 0$.
\end{corollary}


 Finally, the next assertion explains how to compute the defect of a $\frac{\delta}{2}$-even set of nodes on $B$ using a symmetric Casnati--Catanese exact sequence.
 \begin{theorem}[{\cite[Theorem 3.1]{Jozefiak}}]\label{theorem: exact sequence for defect}
 Let $B$ be a nodal complete intersection of a smooth quadric hypersurface $Q$ and a quartic hypersurface in the projective space $\p^4$ and let $w$ be a $\frac{\delta}{2}$-even set of nodes on~$B$. If the sheaf $\F$ associated to $w$ on $Q$ fits into a symmetric Casnati--Catanese resolution \textup{\eqref{eq: main exact sequence}} then the sheaf of ideals~$\I_w(3)$ has the following locally free resolution on $Q$:
 \[
  0\to \Lambda^2\E^{\vee}(-5-\delta)\to \mathrm{sl}(\E)(-1)\to (S^2\E)(3+\delta)\to \I_w(3)\to 0.
 \]
 Here by $\mathrm{sl}(\E)$ we denote the vector bundle of traceless endomorphisms of $\E$.
\end{theorem}

\section{Minimal $0$-even sets of nodes}\label{section: 0-even sets}
 In this section we describe $0$-even sets of nodes on a nodal complete intersection $B$ of a smooth quadric hypersurface $Q$ with a quartic  hypersurface in $\p^4$. The idea is to fix a $0$-even set of nodes $w\subset B$, then to consider the sheaf $\F$ associated to $w$ on $Q$, see Section \ref{subsection: even sets of nodes}, and finally, to construct a Casnati--Catanese symmetric resolution~\eqref{eq: main exact sequence} for $\F$.

 Let $w$ be a minimal non-empty $0$-even set of nodes. The number $\nu = |w|/4$ is an integer by Corollary~\ref{corollary: properties of w}. In view of our assumptions and assertions 3, 5, 6 and 7 of Lemma \ref{lemma: properties of F} the dimensions of cohomology groups of the sheaf $\F(n)$ are as follows (note that the third cohomology group of $\F(n)$ vanishes since the support of the sheaf is a surface):
   \begin{center}
  \begin{longtable}{|c|c|c|c|c|c|c|}
  \caption{\label{table: cohomology groups of F in 0-even case}}\\
  \hline
    & $h^i(Q, \F(-2))$ & $h^i(Q, \F(-1))$ & $h^i(Q,\F)$ & $h^i(Q, \F(1))$ &$h^i(Q, \F(2))$ & $h^i(Q, \F(3))$  \\
  \hline
   $i=2$& $30 - \nu$ & $14 - \nu$  & $6 - \nu$   & $0$ & $0$ & $0$  \\
  \hline
  $i=1$& $0$ &$0$ & $0$ & $0$ & $0$ & $0$ \\
  \hline
  $i=0$& $0$ &$0$ & $0$ & $6 - \nu$ & $14 - \nu$ & $30 - \nu$\\
  \hline
  \end{longtable}
 \end{center}
 This table implies the following property of the sheaf $\F$.
 \begin{lemma}\label{lemma: w is 0-even then F3 is globally generated}
  Let $B$ be a nodal complete intersection of a smooth quadric hypersurface $Q$ and a quartic hypersurface in the projective space $\p^4$, let $w$ be a minimal non-empty $0$-even set of nodes on the surface~$B$ and let $\F$ be the sheaf associated with $w$ on $Q$. Then the sheaf $\F(3)$ is generated by its global sections.
 \end{lemma}
 \begin{proof}
  The table above implies that cohomology group $H^i(Q,\F(3-i))$ vanishes for all $1\leqslant i\leqslant 3$. Thus, the result follows by~\cite[Lecture 14]{Mumford}.
 \end{proof}
 Denote by $x$ the dimension of the cohomology group $H^1(Q,\Spin\otimes\F)$.
 Then by assertions 4 and 8 of Lemma~\ref{lemma: properties of F} the dimensions of cohomology groups of the sheaf $\Spin\otimes \F(n)$ are as follows:
   \begin{center}
  \begin{longtable}{|c|c|c|c|c|c|c|}
  \caption{\label{table: cohomology groups of FxS in 0-even case}}\\
  \hline
     & $h^i(Q, \Spin\otimes\F(-1))$ & $h^i(Q,\Spin\otimes\F)$ & $h^i(Q, \Spin\otimes\F(1))$ &$h^i(Q, \Spin\otimes\F(2))$ & $h^i(Q, \Spin\otimes\F(3))$  \\
  \hline
   $i=2$&  $40 - 2\nu$  & $16 - 2\nu + x$   & $x$ & $0$ & $0$  \\
  \hline
  $i=1$& $0$ & $x$ & $2x+2\nu-8$ & $x$ & $0$ \\
  \hline
  $i=0$& $0$ & $0$ & $x$ &  $16 - 2\nu + x$ & $40 - 2\nu$\\
  \hline
  \end{longtable}
 \end{center} 
 Denote by $\F_B$ the sheaf associated to $w$ on $B$, then $\F = \iota_*\F_B$.
 By Serre duality on the surface $B$ there is an isomorphism $\Ext_B^1(\F_B\otimes\iota_*(\Spin(1)),\OP_B(K_B))\cong H^1(B,\F_B\otimes\iota^*(\Spin(1)))$ where $K_B$ is the canonical class on the surface $B$. This isomorphism and the cup product induce the following perfect pairing:
 \begin{equation}\label{eq: Serre for even sets of nodes}
  H^1(B,\F_B\otimes\iota_*(\Spin(1)))\otimes H^1(B,\F_B\otimes\iota_*(\Spin(1)))
  \to H^2(B,\OP_B(K_B)) \cong \CC.
 \end{equation}
 The last isomorphism is due to Serre duality on the surface $B$.
 Since the cup product is a graded-commutative operation this pairing is alternating. 
 
 Denote by $W$ a maximal isotropic subspace of $H^1(B,\F_B\otimes\iota_*(\Spin(1)))\cong H^1(Q, \Spin\otimes\F(1))$. Then the dimension of $W$ equals $x+\nu - 4$. Denote by $\E_{spin}$ the following vector bundle:
 \begin{equation}\label{eq: E_spin in 0-even case}
 \E_{spin} = \big( H^0(Q,\Spin\otimes \F(2))\otimes \Spin(-2)\big) \oplus \big( W\otimes \Spin(-1)\big). 
 \end{equation}
 Then $\E_{spin}$ is isomorphic to~\mbox{$\Spin(-2)^{\oplus x}\oplus \Spin(-1)^{\oplus x+\nu - 4}$} and by Lemma \ref{lemma: map from spin} applied several times to $\F(1)$ and~$\F(2)$ there is a map of vector bundles:
 \[
   \varphi_{spin}\colon \E_{spin} \to \F
 \]
 such that the map $H^1(\Id_{\Spin}\otimes\varphi_{spin}(n))\colon H^1(Q,\Spin\otimes\E_{spin}(n)) \to H^1(Q,\Spin\otimes\F(n))$ is an isomorphism for $n=2$ and an embedding into the cohomology group $H^1(Q,\Spin\otimes\F(1))$ whose image equals $W$.
 
 Denote by $A$ the graded algebra $\bigoplus_{n\in \ZZ} H^0(Q,\OP_Q(n))$.
 Then $M = \bigoplus_{n\in \ZZ} H^0(Q,\F(n))/ \Im(\varphi_{spin}(\ZZ))$ is a graded $A$-module; here by $\Im(\varphi_{spin}(\ZZ))$ we mean the following graded submodule:
 \[
  \Im(\varphi_{spin}(\ZZ)) = \bigoplus_{n\in \ZZ}\left(H^0(\varphi_{spin}(n))(H^0(Q,\E_{spin}(n)))\right)\subset \bigoplus_{n\in \ZZ} H^0(Q,\F(n)).
 \]
 Since the cohomology group $H^0(Q,\F(n))$ vanishes for all $n\leqslant 0$ and by Lemma \ref{lemma: w is 0-even then F3 is globally generated} the module $M$ is generated by elements in~\mbox{$H^0(Q,\F(1))\oplus H^0(Q,\F(2))\oplus H^0(Q,\F(3))$}. 
 
 Choose the minimal set of graded generators of $M$ (i.e. any element from the set of generators lies in a graded component of $A$) and assume that this set contains exactly $m_i$ elements of $H^0(Q,\F(i))$ for $i = 1,2$ and~$3$. Note that by construction we get $m_1 = h^0(Q,\F(1)) = 6-\nu$. Denote by $\E_{lin}$ the following vector bundle:
 \begin{equation}\label{eq: E_lin in 0-even case}
  \E_{lin} = \OP_Q(-1)^{\oplus 6-\nu}\oplus \OP_Q(-2)^{\oplus m_2}\oplus \OP_Q(-3)^{\oplus m_3}.
 \end{equation}
 Then the set of generators we have choosen defines a map from $\E_{lin}$ to~$\F$:
 \[
  \varphi_{lin}\colon \E_{lin}\to \F.
 \]
 Denote by $\E$ and $\varphi$ the vector bundle $\E_{spin}\oplus \E_{lin}$ and the map $\varphi = \varphi_{spin}\oplus\varphi_{lin}$.
 \begin{lemma}\label{lemma: properties of varphi 0-even sets of nodes}
  Let $\E = \E_{spin}\oplus \E_{lin}$ be the direct sum of vector bundles defined in \textup{\eqref{eq: E_spin in 0-even case}} and \textup{\eqref{eq: E_lin in 0-even case}} and let the map $\varphi=\varphi_{spin}\oplus\varphi_{lin}\colon \E\to \F$ be as defined above. Then $\varphi$ is surjective and satisfies the following properties:
  \begin{enumerate}
   \item[\textup{(1)}] the map $H^0(\varphi(n))\colon H^0(Q,\E(n)) \to H^0(Q,\F(n))$ is surjective for all $n\in \ZZ$ and this map is an isomorphism for~\mbox{$n=1$};
   \item[\textup{(2)}] the map $H^1(\Id_{\Spin}\otimes\varphi(n))\colon H^1(Q,\Spin\otimes\E(n)) \to H^1(Q,\Spin\otimes\F(n))$ is an isomorphism for $n=2$, it is an embedding into the cohomology group $H^1(Q,\Spin\otimes\F(1))$ whose image equals $W$, and it is zero for all $n\ne 1$ or $2$.
  \end{enumerate}
 \end{lemma}
 \begin{proof}
  One has $\E = \E_{spin}\oplus\E_{lin}$ where $\E_{spin}$ and $\E_{lin}$ are vector bundles constructed above.
  The assertion~\textup{(2)} holds by construction of the vector bundle $\E_{spin}$ since the bundle $\Spin\otimes\E_{lin}$ has no intermediate cohomology groups. The assertion \textup{(1)} holds since we choose the vector bundle $\E_{lin}$ in the way that the map~\mbox{$H^0(Q,\E_{lin}(n)) \to H^0(Q,\F(n))/ H^0(Q,\E_{spin}(n))$} is surjective for all $n\in\ZZ$. 
  
  Finally, the fact that $\varphi$ is sujective follows from assertion \textup{(1)}.
 \end{proof}
 Denote by $\K$ the kernel of the surjective map $\varphi$ and by $\Phi$ its embedding into $\E$. Then by Lemma \ref{lemma: properties of varphi 0-even sets of nodes} we have the following exact sequence:
 \begin{equation}\label{eq: CC resolution in 0-even case}
  0\to \K\xrightarrow{\Phi} \E\xrightarrow{\varphi} \F \to 0.
 \end{equation}
 Further we study properties of the sheaf $\K$. We start with the following assertion.
 \begin{proposition}\label{proposition: bundle K in 0-even case}
  Let $\E=\Spin(-2)^{\oplus x}\oplus \Spin(-1)^{\oplus x+\nu - 4}\oplus \OP_Q(-1)^{\oplus 6 - \nu}\oplus \OP_Q(-2)^{\oplus m_2}\oplus \OP_Q(-3)^{\oplus m_3}$, where numbers $m_2$ and $m_3$ are choosen in a minimal way and let the map $\varphi\colon \E\to \F$ be as defined above. Then~\mbox{$m_3 = 0$} and $\K = \Ker(\varphi)$ is isomorphic to $\E^{\vee}(-4)$.
 \end{proposition}
 \begin{proof}
  Since the projective dimension of the sheaf $\F$ equals $1$ by Lemma \ref{lemma: properties of F}(1) we deduce that $\K$ is a vector bundle. By construction of the vector bundle $\E$ one has $H^1(Q,\E(n)) = H^2(Q,\E(n)) = 0$ for all $n\in\ZZ$. Thus, by Lemma \ref{lemma: properties of varphi 0-even sets of nodes} we conclude that $H^0(Q,\K(n)) = 0$ for all $n\leqslant 1$ and that $H^1(Q,\K(n)) = H^2(Q, \K(n)) = 0$ for all~$n\in \ZZ$. Thus, by Lemma \ref{lemma: properties of Spin}(4) the bundle $\K$ is isomorphic to the following:
  \begin{equation}\label{eq: K for 0-even sets}
   \K = \bigoplus_{i\geqslant 1} \Spin(-i)^{\oplus s_i} \oplus \bigoplus_{j\geqslant 2} \OP_Q(-j)^{\oplus n_j},
  \end{equation}
  where $s_i$ and $n_j$ are non-negative integers for $i\geqslant 1$ and $j\geqslant 2$. By Lemma \ref{lemma: properties of Spin}(5) this implies the equalities~\mbox{$h^1(Q, \Spin\otimes\K(n+1))  = h^2(Q, \Spin\otimes\K(n))$} for all~$n\in \ZZ$. 
  
  Consider the long exact sequence of cohomology groups for the exact sequence \eqref{eq: CC resolution in 0-even case} tensored by $\Spin(n)$:
  \begin{multline*}
   0\to H^0(Q, \K\otimes\Spin(n)) \to H^0(Q, \E\otimes\Spin(n))  \to H^0(Q, \F\otimes\Spin(n))  \to H^1(Q, \K\otimes\Spin(n)) \to  H^1(Q, \E\otimes\Spin(n)) \to \\ \to H^1(Q, \F\otimes\Spin(n)) \to H^2(Q, \K\otimes\Spin(n)) \to H^2(Q, \E\otimes\Spin(n)) \to H^2(Q, \F\otimes\Spin(n)) \to \dots
  \end{multline*}
  From this exact sequence we see that $H^0(Q,\K\otimes\Spin(n))$ and $H^1(Q, \K\otimes\Spin(n))$ vanish for all $n\leqslant 0$. The groups~\mbox{$H^0(Q,\K\otimes\Spin(1))$} and $H^0(Q,\K\otimes\Spin(2))$ also vanish by \eqref{eq: K for 0-even sets}; thus, if we fix $n = 1$ then by Lemma \ref{lemma: properties of varphi 0-even sets of nodes} we get $h^1(Q, \K\otimes\Spin(1)) = x$. Thus, we conclude  that $s_1 = x$ in \eqref{eq: K for 0-even sets}.
  
  If we fix $n = 2$ the long exact sequence of cohomology groups implies that $h^1(Q, K\otimes\Spin(2)) = x+\nu-4$. Moreover, one can see that the group $H^2(Q,\K\otimes \Spin(n))$ vanishes for $n\geqslant 2$. Thus, $s_2 = x+\nu - 4$ and $s_i = 0$ for all $i\geqslant 3$ in \eqref{eq: K for 0-even sets}.
  
  Since the ranks of bundles $\E$ and $\K$ are equal and by Lemma \ref{lemma: properties of varphi 0-even sets of nodes} we get the following equalities:
  \begin{align*}
   4x + 2\nu - 8 + \sum n_j = \rk(\K) &= \rk(\E) = 4x + \nu - 2 + m_2+m_3;\\
   4x -\nu +n_2 + 14 = h^0(Q,\K(2)) + h^0(Q,\F(2)) &= h^0(Q, \E(2)) = 4x - \nu +m_2+14;\\
   20 x + 3\nu+ 14 + 5n_2 +n_3 = h^0(Q,\K(3)) + h^0(Q,\F(3)) &= h^0(Q, \E(3)) = 20 x +2 \nu -20 + 5m_2 +m_3.
  \end{align*}
  This implies that $n_2 = m_2$ and $n_3 = m_3 -\nu + 6$ and $n_j = 0$ for all $j\geqslant 3$. Finally, since the map from $\K$ to $\E$ is an embedding and groups $\Hom_Q(\Spin(-1), \OP_Q(-3))$, $\Hom_Q(\Spin(-2), \OP_Q(-3))$ and $\Hom_Q(\OP_Q(-2), \OP_Q(-3))$ vanish and since the choice of numbers $m_i$ was minimal we conclude by Lemma \ref{lemma: excluding the isomorphism} that $m_3 = 0$. This implies the result.
 \end{proof}
 By Lemma \ref{lemma: excluding the isomorphism} there could be many Casnati--Catanese resolutions for a given sheaf $\F$. Thus, we restrict our search only for minimal resolutions of this type in the following sense. We say that the map of vector bundles $\Phi\colon \K\to \E$ is {\itshape minimal} if for any subbundles $\K_0\subset \K$ and~$\E_0\subset \E$ the composition $\mathrm{pr}_{\E_0}\circ \Phi|_{\K_0} \colon \K_0\to\E_0$ is not an isomorphism where $\Phi|_{\K_0}$ is the restriction of $\Phi$ to the direct summand of $\K$ and $\mathrm{pr}_{\E_0}$ is the projection of $\E$ to its direct summand. This condition implies the following condition for parameters of $\E$.
 \begin{lemma}\label{lemma: restrictions for 0-even sets}
  Let $B$ be a nodal complete intersection of a smooth quadric hypersurface $Q$ and a quartic hypersurface in $\p^4$, let $w$ be a minimal $0$-even set of $4\nu$ nodes on $B$ and let the embedding $\Phi\colon \K\to \E$ be a minimal map satisfying properties listed in Proposition \textup{\ref{proposition: bundle K in 0-even case}}. Then we have the following constraint for the numbers $m_2$, $x$ and $\nu$:
  \begin{enumerate}
   \item[\textup{(1)}] $8-2\nu \leqslant 2x \leqslant \max\{0,5 - \nu\}$;
   \item[\textup{(2)}] $m_2 \leqslant \nu - 3$;
   \item[\textup{(3)}] $4\leqslant \nu\leqslant 6$.
  \end{enumerate}
 \end{lemma}
 \begin{proof}
  The groups $\Hom_Q(\Spin(-1), \V)$ vanish for $\V = \Spin(-2)$ and $\OP_Q(-2)$. Since $\Hom_Q(\Spin(-1),\Spin(-1))$ is generated by the identity, the minimality of the resolution of $\F$ implies that $\Phi(\Spin(-1)^{\oplus x})$ is a subsheaf of the bundle $\OP_Q(-1)^{\oplus 6 - \nu}$. Therefore, $2x \leqslant 6-\nu$. Moreover, in the case of equality by Lemma \ref{lemma: excluding the embedding of full rank} the surface~$B$ is reducible; thus, $2x \leqslant 5 - \nu$. Finally, the number $x + \nu -4$ also should be non-negative; thus, $2x\geqslant 8-2\nu$. Thus, we get the assertion \textup{(1)}.
  
  The assertion \textup{(2)} follows from the irreducibility of $B$ and the fact that $\Phi(\Spin(-1)^{\oplus x}\oplus \OP_Q(-2)^{\oplus m_2})$ is a subsheaf of $\OP_Q(-1)^{\oplus 6 - \nu}\oplus \Spin (-1)^{\oplus x+\nu-4}$. Finally, the assertion \textup{(3)} is an implication of the assertion \textup{(1)}, the fact that the number $6 - \nu$ is non-negative and Lemma \ref{lemma: JR restriction on cohomology groups of F}.
 \end{proof}
 
  Now we can describe all minimal $0$-even sets of nodes $w$ which can arise on the  intersection of a smooth quadric hypersurface $Q$ and a quartic hypersurface. Moreover, Theorem \ref{theorem: exact sequence for defect} allows us to compute defects~$d(w)$ for them.
 \begin{corollary}\label{corollary: CC-resolutions in 0-even case}
  Let $B$ be a nodal complete intersection of a smooth quadric hypersurface $Q$ and a quartic hypersurface in the projective space $\p^4$ and let $w$ be a minimal non-empty $0$-even set of nodes on the surface~$B$. Then~\mbox{$B = \left\{\det\left(\Phi\colon \E^{\vee}(-4)\to \E\right) = 0 \mid \Phi\in H^0(Q, (S^2\E)(4))\right\}$} and the following cases are possible:
  \begin{align*}
   \text{\textup{(E1)}} \hspace{1cm} &|w| = 16 \text{ and } \E =\OP_Q(-1)^{\oplus 2}\oplus \OP_Q(-2), && d(w) = 1;\\
   \text{\textup{(E2)}} \hspace{1cm} &|w| = 20 \text{ and } \E = \Spin(-1)\oplus \OP_Q(-1), && d(w) = 0;\\
   \text{\textup{(E3)}} \hspace{1cm} &|w| = 24 \text{ and } \E = \Spin(-1)^{\oplus 2}, && d(w) = 0.
  \end{align*}
  Moreover, if one takes a general element $\Phi$ of $ H^0(Q, (S^2\E)(4))$ for $\E$ as in \textup{(E1 -- E3)} then $B = \{\det(\Phi) = 0\}$ is a nodal complete intersection of $Q$ and a quartic hypersurface in $\p^4$ and $|\Sing(B)| = 16$, $20$ and $24$ respectively.
 \end{corollary}
 \begin{proof}
  Let $w$ be a minimal non-empty $0$-even set of $4\nu$ nodes on the surface~$B$. Consider minimal Casnati--Catanese resolution \eqref{eq: CC resolution in 0-even case} as in Proposition \ref{proposition: bundle K in 0-even case}. By Lemma \ref{lemma: restrictions for 0-even sets} we conclude $\nu = 4$, $5$ or $6$. We study each of these cases below.
  
  If $\nu = 4$ then $|w|=16$. Lemma \ref{lemma: restrictions for 0-even sets} implies that $x=0$ and $m_2\leqslant 1$. Then in fact two cases are possible~\mbox{$m_2 = 0$} and $m_2 = 1$; thus, the vector bundle $\E$ equals either $ \OP_Q(-1)^{\oplus 2}$ or $ \OP_Q(-1)^{\oplus 2}\oplus \OP_Q(-2)$. However, for any element $\Phi'\in H^0(Q,(S^2\E')(4))$ where $\E' = \OP_Q(-1)^{\oplus 2}$ there exists an element $\Phi\in H^0(Q,(S^2\E)(4))$ where~\mbox{$\E = \OP_Q(-1)^{\oplus 2}\oplus \OP_Q(-2)$} such that $\det(\Phi)$ and $\det(\Phi')$ defines same divisor $B$ in $Q$. Thus, the case~\textup{(E1)} describes all possible $0$-even sets of $16$ nodes on $B$. Since $H^1(Q,(S^2\E^{\vee})(-4)) = 0$ then by Lemma~\ref{lemma: CC plus H1 vanishes implies symmetricity} there $\F$ fits the necessary symmetric Casnati--Catanese resolution.
  
  If $\nu = 5$ then $|w| = 20$. Lemma \ref{lemma: restrictions for 0-even sets} implies $x=0$ and $m_2\leqslant 2$. Since the resolution is minimal \eqref{eq: CC resolution in 0-even case}  and the group $\Hom_Q(\Spin(-2), \OP_Q(-2))$ vanishes, the only direct summand of $\K$ which maps with non-zero image to $\OP_Q(-2)^{\oplus m_2}$ is $\OP_Q(-3)$. Therefore, $m_2 \leqslant 1$ and the case of the equality is excluded by  by Lemma \ref{lemma: excluding the embedding of full rank} since $B$ is an irreducible surface. Thus, we get the vector bundle from the case \textup{(E2)}. By Lemma \ref{lemma: cases E2 and E3 are symmetric} below the resolution \eqref{eq: CC resolution in 0-even case} is symmetric in this case.
  
  If $\nu = 6$ then $|w| = 24$.  Lemma \ref{lemma: restrictions for 0-even sets} implies $x=0$ and $m_2\leqslant 3$. Since the resolution is minimal and the group $\Hom_Q(\Spin(-2), \OP_Q(-2))$ vanishes we get that $\Phi(\K)$ is a subsheaf of $\Spin(-1)^{\oplus 2}$. Thus, $m_2 = 0$ and we get the case \textup{(E3)}. By Lemma \ref{lemma: cases E2 and E3 are symmetric} below the resolution \eqref{eq: CC resolution in 0-even case} is symmetric in this case.
  
  The defect in all the cases is computed by Theorem \ref{theorem: exact sequence for defect}. The last assertion follows from the modification of Bertini theorem \cite[Lemma 4]{Barth} since the vector bundle $S^2\E$ is globally generated for bundles $\E$ as in cases \textup{(E1 -- E3)}.
 \end{proof}
 It remains to prove the symmetricity of the Casnati--Catanese resolutions in the cases \textup{(E2)} and \textup{(E3)}.
 \begin{lemma}\label{lemma: cases E2 and E3 are symmetric}
  Let $B$ be a nodal complete intersection of a smooth quadric hypersurface $Q$ and a quartic hypersurface in the projective space $\p^4$ and let $w$ be a minimal non-empty $0$-even set of nodes on the surface~$B$. Assume that the sheaf $\F$ associated with $w$ on $Q$ fits the Casnati--Catanese resolution \eqref{eq: CC resolution in 0-even case} where $\E$ is as in the cases \textup{(E2)} or \textup{(E3)} and the map $\varphi$ is as in Lemma \textup{\ref{lemma: properties of varphi 0-even sets of nodes}}. Then the Casnati--Catanese resolution is symmetric.
 \end{lemma}
 \begin{proof}
  Apply the functor $\HomSh_Q(-,\OP_Q(-4))$ to the Casnati--Catanese resolution \eqref{eq: CC resolution in 0-even case}. We get the following exact sequence:
  \begin{equation}\label{eq: dual CC}
   0\to\E^{\vee}(-4) \xrightarrow{\Phi^{\vee}} \E\xrightarrow{\deltamap} \ExtSh_Q^1(\F,\OP_Q(-4))\to 0.
  \end{equation}
  Here $\deltamap$ is the connecting homomorphism of the long exact sequence. By Lemma \ref{lemma: properties of F}(2) there exists an isomorphism    $s\colon \F\to \ExtSh_Q^1(\F,\OP_Q(-4))$. It suffices to show that $s\circ\varphi = \deltamap$; i.e. $s$ induces an isomorphism of the maps $\deltamap$ and $\varphi$.
  
  First consider the case \textup{(E2)} that is $\E = \Spin(-1)\oplus \OP_Q(-1)$. By construction the map $\varphi$ is the direct sum of maps $\varphi_{spin}\colon \Spin(-1)\to \F$ and $\varphi_{lin}\colon \OP_Q(-1)\to \F$. Denote by $\deltamap_{spin}$ and $\deltamap_{lin}$ the components of the map $\deltamap$:
  \begin{align*}
   \deltamap_{spin}\colon \Spin(-1)\to \ExtSh_Q^1(\F,\OP_Q(-4)); && \deltamap_{lin}\colon \OP_Q(-1)\to \ExtSh_Q^1(\F,\OP_Q(-4)).
  \end{align*}
  Here $\deltamap = \deltamap_{spin}\oplus\deltamap_{lin}$.
  Tensor the exact sequences \eqref{eq: CC resolution in 0-even case} and \eqref{eq: dual CC} by $\Spin(-1)$ and consider the long exact sequence of cohomology groups:
\[\begin{tikzcd}[ampersand replacement=\&, column sep=2em]
	0 \& {H^1(Q,\E\otimes\Spin(1))} \& {H^1(Q, \F\otimes\Spin(1))} \& {H^2(Q,\E^{\vee}(-4)\otimes\Spin(-1))} \& 0 \\
	0 \& {H^1(Q,\E\otimes\Spin(1))} \& {H^1(Q, \ExtSh_Q^1(\F,\OP_Q(-4))\otimes\Spin(1))} \& {H^2(Q,\E^{\vee}(-4)\otimes\Spin(-1))} \& 0
	\arrow[from=1-4, to=1-5]
	\arrow[from=2-4, to=2-5]
	\arrow[from=2-3, to=2-4]
	\arrow[from=1-3, to=1-4]
	\arrow["\varphi", from=1-2, to=1-3]
	\arrow["\deltamap", from=2-2, to=2-3]
	\arrow[from=1-1, to=1-2]
	\arrow[from=2-1, to=2-2]
	\arrow["s", from=1-3, to=2-3]
\end{tikzcd}\]
  If one applies the Serre duality to the upper exact sequence we get the dualisation of the lower one. By Lemma~\ref{lemma: properties of varphi 0-even sets of nodes} we have $H^1(\varphi)\colon H^1(Q,\E\otimes\Spin(1))\to H^1(Q, \F\otimes\Spin(1))$ is the embedding of $H^1(Q,\E\otimes\Spin(1))$ to the isotropic subspace $W$ of $ H^1(Q, \F\otimes\Spin(1))$. 
  Since \eqref{eq: dual CC} is the dualisation of \eqref{eq: CC resolution in 0-even case} the map 
  \[
  \deltamap\colon H^1(Q, \E\otimes\Spin(1))\to H^1(Q, \ExtSh_Q^1(\F,\OP_Q(-4))\otimes\Spin(1))
  \]
  is also an embedding to an isotropic subspace $s(W)$ of $H^1(Q, \ExtSh_Q^1(\F,\OP_Q(-4))\otimes\Spin(1))$. Thus, the maps~$\varphi_{spin}$ and $\deltamap_{spin}$ are isomorphic by Lemma \ref{lemma: map from spin}.
  
  Now tensor the exact sequences \eqref{eq: CC resolution in 0-even case} and \eqref{eq: dual CC} by $\OP_Q(1)$.
The maps $\varphi_{lin}$ and $\deltamap_{lin}$ are defined by the images of the cohomology group $H^0(Q,\E(1))\cong \CC$ under the maps $\varphi$ and $\deltamap$ in the groups $H^0(Q,\F(1))$ and~\mbox{$H^0(Q, \ExtSh_Q^1(\F,\OP_Q(-3)))$} respectively. Since $H^0(Q,\F(1))\cong H^0(Q, \ExtSh_Q^1(\F,\OP_Q(-3)))\cong \CC$ there is an isomorphism between maps $\varphi_{lin}$ and $\deltamap_{lin}$. Thus, the map $s$ induces an isomorphism of the maps $\varphi$ and $\deltamap$ and the map $\Phi$ is self-dual in the case \textup{(E2)}.
  
  The case \textup{(E3)} is analogous. Thus, the proof is finished.
 \end{proof}

 \section{Minimal $\frac{1}{2}$-even sets of at most 20 nodes}\label{section: 1/2-even sets less than 20}
 In this section we study $\frac{1}{2}$-even sets of nodes on a nodal complete intersection $B$ of a smooth quadric hypersurface $Q$ with a quartic  hypersurface in $\p^4$. Fix a $\frac{1}{2}$-even set of nodes $w\subset B$ such that $|w| \leqslant 20$ and consider the sheaf $\F$ associated to $w$ on $Q$ as defined in Section \ref{subsection: even sets of nodes}. Our goal is to construct a Casnati--Catanese symmetric resolution for the sheaf $\F$. 

  In the case when $|w|\leqslant 20$ we have the following condition on the cohomology groups of the sheaf $\F$.
 \begin{lemma}\label{lemma: h1(F) vanishes in case mu less than 3}
  Let $B$ be a nodal complete intersection of a smooth quadric hypersurface $Q$ and a quartic hypersurface in the projective space $\p^4$ and let $w$ be a minimal non-empty $\frac{1}{2}$-even set of nodes on the surface~$B$. If $|w|\leqslant 20$ then $H^1(Q,\F(1)) = 0$.
 \end{lemma}
 \begin{proof}
   By assertions (3) and (7) of  Lemma \ref{lemma: properties of F} one has $h^2(Q,\F(1)) = h^0(Q,\F(1))$ and $\chi(\F(1)) = 5-|w|/4$. Moreover, by Lemma \ref{lemma: JR restriction on cohomology groups of F} we get that $h^0(Q,\F(1)) \leqslant 2$. This implies that if $|w| = 4$ then~\mbox{$H^1(Q,\F(1))$} vanishes.
 
  Recall that $\widetilde{B}$ is the minimal resolution of singularities of the nodal surface $B$ and that $E_w$ is the sum of classes of exceptional divisors over points in the set $w\subset \Sing(B)$ on $\widetilde{B}$. Let $\F_B$ be the sheaf associated to $w$ on $B$. Then $\F = \iota_*\F_B$ and $h^1(Q, \F(1)) = h^1(B,\F_B(1))$ since the closed embedding $\iota\colon B\to Q$ is an affine morphism.

  By Corollary \ref{corollary: properties of w} if $|w|\leqslant 20$ then $|w| = 4$ or $12$ or $20$.  If there is a non-zero section of the sheaf~$\F_B(1)$ then it defines a non-zero element $s$ in the group $H^0\left(\widetilde{B}, \OP_{\widetilde{B}}\left(\frac{H-E_w}{2}\right)\right)$. Then $s^2$ is an element in the cohomology group $H^0(\widetilde{B}, \OP_{\widetilde{B}}(H - E_w))$ so the $\frac{1}{2}$-even set of nodes $w$ lies on the intersection of the surface $B$ with a hyperplane section $H$ in $\p^4$.
  
  Assume that  $h^0(B,\F_B(1)) = 2$ and $|w| = 12$ or $20$. Then there exists a two-dimensional space of hyperplanes which contain the set of nodes $w$. Let $H_1$ and $H_2$ be two elements in this space. Thus, $w$ lies on the intersection $H_1\cap H_2\cap Q$ which is a conic. Then we get a contradiction by Lemma \ref{lemma: lines and conics on B containing 1/2-even sets}(3). Thus, if $|w| = 12$ or $20$ then $h^0(B,\F_B(1))\leqslant 1$. Therefore, if $|w| = 12$ then~\mbox{$H^1(Q,\F(1))$} vanishes.
 
  Let $w$ be an even set of $20$ nodes. By the previous argument we conclude that $h^0(B, \F_B(1)) \leqslant 1$. Assume that this is an equality. Consider the multiplication map:
  \[
   m\colon H^0(B, \F_B(1))\otimes H^0(B,\OP_B(1)) \to H^0(B,\F_B(2))
  \]
  On the one hand, since $h^0(B, \F_B(1)) = 1$ the map $m$ is an embedding; otherwise there exists a non-zero section in $H^0(B,\OP_B(1))$ which vanishes on $B$. On the other hand, in view of assertions (6) and (7) of Lemma \ref{lemma: properties of F} the dimension of the cohomology group $H^0(B,\F_B(2))$ equals $4$ which is less that $h^0(B,\OP_B(1)) = 5$. Thus, we get a contradiction; therefore, both numbers $h^0(B, \F_B(1))$ and  $h^1(B, \F_B(1))$ also vanish by Lemma \ref{lemma: properties of F}(3). Thus, so does $h^1(Q,\F(1))$.
 \end{proof}
 By Corollary \ref{corollary: properties of w} we get that $|w|\equiv 4 \mod 8$. Let $\mu$ be an integer such that $|w| = 8\mu+4$. Then by assertions (3), (5), (6) and (4) of Lemma \ref{lemma: properties of F} and by Lemma~\ref{lemma: h1(F) vanishes in case mu less than 3} we get the following table of cohomology groups of $\F$ in the case when $|w|\leqslant 20$.
  \begin{center}
  \begin{longtable}{|c|c|c|c|c|c|c|}
  \caption{\label{table: cohomology groups of F in 1/2-even case less than 28 nodes}}\\
  \hline
   & $h^i(Q, \F(-1))$ & $h^i(Q,\F)$ & $h^i(Q, \F(1))$ &$h^i(Q, \F(2))$ & $h^i(Q, \F(3))$ & $h^i(Q,\F(4))$ \\
    \hline
    $i=2$& $20-2\mu$ & $8 - 2\mu$     & $2-\mu$ & $0$ & $0$ &$0$ \\
  \hline
  $i=1$& $0$ &$0$  & $0$ & $0$ & $0$ &$0$\\
  \hline
  $i=0$& $0$ &$0$  & $2 - \mu$ & $8 - 2\mu$ & $20 - 2\mu$& $40-2\mu$\\
  \hline
  \end{longtable}
 \end{center}
  This table implies the following property of the sheaf $\F$.
 \begin{lemma}\label{lemma: w is half-even then F4 is globally generated}
  Let $B$ be a nodal complete intersection of a smooth quadric hypersurface $Q$ and a quartic hypersurface in the projective space $\p^4$, let $w$ be a minimal non-empty $\frac{1}{2}$-even set of nodes on the surface~$B$ and let $\F$ be the sheaf associated with $w$ on $Q$. Then the sheaf $\F(4)$ is generated by its global sections.
 \end{lemma}
 \begin{proof}
  The table above implies that cohomology group $H^i(Q,\F(4-i))$ vanishes for all $1\leqslant i\leqslant 3$. Thus, the result follows by~\cite[Lecture 14]{Mumford}.
 \end{proof}
 Denote by $x$ and $y$ dimensions of cohomology groups $H^1(Q,\Spin\otimes\F)$ and $H^1(Q,\Spin\otimes\F(1))$ respectively.
 Then by Lemma \ref{lemma: properties of F}(8) dimensions of cohomology groups of the sheaf $\Spin\otimes \F(n)$ are as follows:
   \begin{center}
  \begin{longtable}{|c|c|c|c|c|c|c|}
  \caption{\label{table: cohomology groups of FxS in  1/2-even case less than 28 nodes}}\\
  \hline
     & $h^i(Q, \Spin\otimes\F(-1))$ & $h^i(Q,\Spin\otimes\F)$ & $h^i(Q, \Spin\otimes\F(1))$ &$h^i(Q, \Spin\otimes\F(2))$ & $h^i(Q, \Spin\otimes\F(3))$  \\
  \hline
   $i=2$&  $56-4\mu$  & $24 - 4\mu + x$   & $8-4\mu-x+y$ & $x$ & $0$  \\
  \hline
  $i=1$& $0$ & $x$ & $y$ & $y$ & $x$ \\
  \hline
  $i=0$& $0$ & $0$ & $x$ &  $8-4\mu-x+y$ & $24 - 4\mu + x$\\
  \hline
  \end{longtable}
 \end{center} 
 Denote by $\E_{spin}$ the following vector bundle:
 \begin{equation}\label{eq: E_spin in half-even case}
 \E_{spin} = \big( H^0(Q,\Spin\otimes \F(3))\otimes \Spin(-3)\big) \oplus \big( H^0(Q,\Spin\otimes \F(2))\otimes \Spin(-2)\big).
 \end{equation}
  Then $\E_{spin}$ is isomorphic to $\Spin(-3)^{\oplus x}\oplus \Spin(-2)^{\oplus y}$ and by Lemma \ref{lemma: map from spin} applied several times to $\F(2)$ and $\F(3)$ there is a map of vector bundles:
 \[
   \varphi_{spin}\colon \E_{spin} \to \F
 \]
 such that the map $H^1(\Id_{\Spin}\otimes\varphi_{spin}(n))\colon H^1(Q,\Spin\otimes\E_{spin}(n)) \to H^1(Q,\Spin\otimes\F(n))$ is an isomorphism for $n=2$ and $3$.
 
 Denote by $A$ the graded algebra $\bigoplus_{n\in \ZZ} H^0(Q,\OP_Q(n))$.
 Then $M = \bigoplus_{n\in \ZZ} H^0(Q,\F(n))/ \Im(\varphi_{spin}(\ZZ))$ is a graded left $A$-module; here by $\Im(\varphi_{spin}(\ZZ))$  the following graded submodule:
 \[
  \Im(\varphi_{spin}(\ZZ)) = \bigoplus_{n\in \ZZ}\left(H^0(\varphi_{spin}(n))(H^0(Q,\E_{spin}(n)))\right)\subset \bigoplus_{n\in \ZZ} H^0(Q,\F(n)).
 \]
 Since the cohomology group $H^0(Q,\F(n))$ vanishes for all $n\leqslant 0$ and by Lemma \ref{lemma:  w is half-even then F4 is globally generated} the module $M$ is generated by elements in its first, second, third and fourth graded components. 
 
 Choose the minimal set of graded generators of $M$ (i.e. any element from the set of generators lies in a graded component of $M$) and assume that this set contains exactly $m_i$ elements of $H^0(Q,\F(i))$ for $i = 1,2,3$ and~$4$.  Note that by construction we get $m_1 = h^0(Q,\F(1)) = 2-\mu$. Denote by $\E_{lin}$ the following vector bundle:
 \begin{equation}\label{eq: E_lin in half-even case}
  \E_{lin} = \OP_Q(-1)^{\oplus 2-\mu}\oplus \OP_Q(-2)^{\oplus m_2}\oplus \OP_Q(-3)^{\oplus m_3}\oplus \OP_Q(-4)^{\oplus m_4}.
 \end{equation}
  Then the set of generators we have choosen defines a map from $\E_{lin}$ to~$\F$:
 \[
  \varphi_{lin}\colon \E_{lin}\to \F.
 \]
 Denote by $\E$ and $\varphi$ the vector bundle $\E_{spin}\oplus \E_{lin}$ and map $\varphi = \varphi_{spin}\oplus\varphi_{lin}$.  
 \begin{lemma}\label{lemma: properties of varphi 1/2-even sets of nodes less than 28}
  Let $\E = \E_{spin}\oplus \E_{lin}$ be the direct sum of vector bundles defined in \textup{\eqref{eq: E_spin in 0-even case}} and \textup{\eqref{eq: E_lin in 0-even case}} and let the map $\varphi=\varphi_{spin}\oplus\varphi_{lin}\colon \E\to \F$ be as defined above. Then $\varphi$ is surjective and satisfies the following properties:
  \begin{enumerate}
   \item[\textup{(1)}] the map $H^0(\varphi(n))\colon H^0(Q,\E(n)) \to H^0(Q,\F(n))$ is surjective for all $n\in \ZZ$ and this map is an isomorphism for~\mbox{$n=1$};
   \item[\textup{(2)}] the map $H^1(\Id_{\Spin}\otimes\varphi(n))\colon H^1(Q,\Spin\otimes\E(n)) \to H^1(Q,\Spin\otimes\F(n))$ is an isomorphism for $n=2$ and $3$ and it is zero for all $n\ne 2$ or $3$.
  \end{enumerate}
 \end{lemma}
 \begin{proof}
  The proof of this lemma is analogous to one of Lemma \ref{lemma: properties of varphi 0-even sets of nodes}.
 \end{proof}
 Denote by $\K$ the kernel of the map $\varphi$ and by $\Phi$ its embedding into $\E$; we get an exact sequence:
 \begin{equation}\label{eq: CC resolution in 1/2-even case less than 28}
  0\to \K\xrightarrow{\Phi} \E\xrightarrow{\varphi} \F \to 0.
 \end{equation}
 Further we study properties of the sheaf $\K$. We start with the following assertion.
 \begin{proposition}\label{proposition: bundle K in 1/2-even case less than 28}
  Let $\E = \E_{spin}\oplus \E_{lin}$ be the direct sum of vector bundles defined in \textup{\eqref{eq: E_spin in 0-even case}} and \textup{\eqref{eq: E_lin in 0-even case}} and let the map $\varphi=\varphi_{spin}\oplus\varphi_{lin}\colon \E\to \F$ be as defined above. If we denote by $z = 4x+3\mu-2$ then
  \[
   \K = \Spin(-1)^{\oplus x}\oplus \Spin(-2)^{\oplus y} \oplus \OP_Q(-2)^{m_2 - z}\oplus\OP_Q(-3)^{m_3 + z}\oplus \OP_Q(-4)^{\oplus m_4+2-\mu}.
  \]

 \end{proposition}
 \begin{proof}
  Since the projective dimension of $\F$ equals $1$ by Lemma \ref{lemma: properties of F} we get that $\K$ is a vector bundle. By construction of the vector bundle $\E$ cohomology groups $H^1(Q,\E(n))$ and $H^2(Q,\E(n))$ vanish for all $n\in\ZZ$. Lemma \ref{lemma: properties of varphi 1/2-even sets of nodes less than 28} implies that $H^0(Q,\K(n)) = 0$ for all $n\leqslant 1$ and that groups $H^1(Q,\K(n))$ and $H^2(Q, \K(n))$ vanish for all~$n\in \ZZ$. Thus, by \cite[Theorem 3.5]{Ottaviani-Horrocks_criterion} the bundle $\K$ is isomorphic to the following:
  \begin{equation}\label{eq: K for 1/2-even sets less than 28}
   \K = \bigoplus_{i\geqslant 1} \Spin(-i)^{\oplus s_i} \oplus \bigoplus_{j\geqslant 2} \OP_Q(-j)^{\oplus n_j},
  \end{equation}
  where $s_i$ and $n_j$ are non-negative integers for $i\leqslant -1$ and $j\leqslant -2$. By Lemma \ref{lemma: properties of F} this implies the equality~\mbox{$h^1(Q, \Spin\otimes\K(n+1))  = h^2(Q, \Spin\otimes\K(n))$} for all~$n\in \ZZ$. 
  
  Consider the long exact sequence of cohomology groups for the exact sequence \eqref{eq: CC resolution in 1/2-even case less than 28} tensored by $\Spin(n)$:
  \begin{multline*}
   0\to H^0(Q, \K\otimes\Spin(n)) \to H^0(Q, \E\otimes\Spin(n))  \to H^0(Q, \F\otimes\Spin(n))  \to H^1(Q, \K\otimes\Spin(n)) \to  H^1(Q, \E\otimes\Spin(n)) \to \\ \to H^1(Q, \F\otimes\Spin(n)) \to H^2(Q, \K\otimes\Spin(n)) \to H^2(Q, \E\otimes\Spin(n)) \to H^2(Q, \F\otimes\Spin(n)) \to \dots
  \end{multline*}
  From this exact sequence we get that $H^0(Q,\K\otimes\Spin(n))$ and $H^1(Q, \K\otimes\Spin(n))$ vanish for all $n\leqslant 0$. Substituting $n = 0, 1$ and $2$ we get that $s_1 = x$ and $s_2 = y$ by Lemma \ref{lemma: properties of varphi 1/2-even sets of nodes less than 28}.
  
  Consider the long exact sequence of cohomology groups for the exact sequence \eqref{eq: CC resolution in 1/2-even case less than 28} twisted by $2$. It implies that $n_2 = m_2 - 4x-3\mu+2$. Then by the long exact sequence for $\K\otimes \Spin(3)$ we get that $s_3 = 0$.
  
  Consider the long exact sequence of cohomology groups for the exact sequence \eqref{eq: CC resolution in 1/2-even case less than 28} twisted by $3$. It implies that $n_3 = m_3 + 4x+3\mu-2$. Then by the long exact sequence for $\K\otimes \Spin(j)$ we get that $s_j = 0$ for~\mbox{$j\geqslant 4$.}
  
  Computing the rank of $\K$ and considering the  the exact sequence \eqref{eq: CC resolution in 1/2-even case less than 28} twisted by $4$ we get~\mbox{$n_4 = m_4 + 2-\mu$} and $n_j$ vanishes for all $j\geqslant 5$. This finishes the proof.
  \end{proof}
  
  As well as in Section \ref{section: 0-even sets} here it suffices for us to construct minimal Casnati--Catanese resolution. Recall that that the map $\Phi\colon \K\to \E$ is  minimal i.e. for any subbundles $\K_0\subset \K$ and~$\E_0\subset \E$ the composition $\mathrm{pr}_{\E_0}\circ \Phi|_{\K_0} \colon \K_0\to\E_0$ is not an isomorphism where $\Phi|_{\K_0}$ is the restriction of $\Phi$ to the direct summand of $\K$ and $\mathrm{pr}_{\E_0}$ is the projection of $\E$ to its direct summand. This condition implies the following condition for parameters of $\E$.
  
  \begin{lemma}\label{lemma: constraint of minimal map for half-even sets of nodes}
   Let $B$ be a nodal complete intersection of a smooth quadric hypersurface $Q$ and a quartic hypersurface in $\p^4$, let $w$ be a minimal $\frac{1}{2}$-even set of $4\mu+4$ nodes on $B$ and let the embedding $\Phi\colon \K\to \E$ be a minimal map satisfying properties listed in Proposition \textup{\ref{proposition: bundle K in 1/2-even case less than 28}}. Then $x = m_4 = 0$ and $\mu = 1$ or $2$.
  \end{lemma}
  \begin{proof}
   One has $\Hom_Q(\OP_Q(i), \OP_Q(-4)) = 0$ for $i\geqslant -3$.
   By Lemma \ref{lemma: properties of Spin}(2) one has $\Hom_Q(\Spin(j), \OP_Q(-4)) = 0$ for $j\geqslant -3$. Therefore, we get that the only direct summand of $\K$ with a non-zero map to $\OP_Q(-4)^{\oplus m_4}\subset \E$ is $\OP_Q(-4)^{\oplus m_4+2-\mu}\subset \K$. Since the minimal embedding $\Phi$ is of full rank then by Lemma \ref{lemma: excluding the isomorphism} we conclude that $m_4 = 0$.
   
   
   By  Lemma \ref{lemma: properties of Spin}(2) and (3) one has $\Hom_Q(\Spin(-1), \Spin(i)) = 0$ for $i\leqslant -2$ and $\Hom_Q(\Spin(-1),\OP_Q(j)) = 0$ for~\mbox{$j\leqslant -2$}. Therefore, we get that the image of $\Phi|_{\Spin(-1)^{\oplus x}}$ lies in the subbundle $\OP_Q(-1)^{\oplus 2-\mu}\subset \E$. Since the  embedding $\Phi$ is of full rank then $x\leqslant (2-\mu)/2$. Thus, either $\mu = 0$ or $\mu = 1$ or $2$ and in two latter cases $x = 0$. To complete the proof we have to exclude the former case.
   
If $\mu = 0$ then $x\leqslant 1$ by the above argument. By  Lemma \ref{lemma: properties of Spin}(2) one has~\mbox{$\Hom_Q(\OP_Q(-2),\Spin(i)) = 0$} for~\mbox{$i\leqslant -2$} and $\Hom_Q(\OP_Q(-2),\OP_Q(j)) = 0$ for $j\leqslant -3$. Therefore, we get that the image of $\Phi|_{\OP_Q(-2)^{\oplus m_2-z}}$ lies in the subbundle $\OP_Q(-1)^{\oplus 2-\mu}\oplus\OP_Q(-2)^{\oplus m_2}\subset \E$. Since the minimal embedding $\Phi$ is of full rank then by Lemma~\ref{lemma: excluding the isomorphism} we conclude that $m_2-z\leqslant 2-\mu = 2$. If one has $x = 0$ then this implies that $m_2 +2\leqslant 0$ and this contradicts the fact that $m_2$ is a non-negative integer. Therefore, $x = 1$ if $\mu=0$.
   
   Assume that $\mu = 0$ and $x=1$. Then one can easily check that by Lemma \ref{lemma: excluding the isomorphism} the image of the subbundle~\mbox{$\Spin(-1)^{\oplus x}\oplus \OP_Q(-2)^{\oplus m_2 - 4x +2}$} of the vector bundle $\K$ under the minimal embedding $\Phi$ lies in the subbundle $\OP_Q(-1)^{\oplus 2-\mu}$ of $\E$. Therefore, $m_2 = 2x + m_2-4x + 2\leqslant 2-\mu = 2$. Denote by $\mathrm{pr}$ the projection from $\E$ to the direct sum of all its  direct summands except $\OP_Q(-1)^{\oplus 2}$:
   \[
    \mathrm{pr}\colon \E\to \Spin(-3)\oplus\Spin(-2)^{\oplus y}\oplus \OP_Q(-2)^{\oplus 2}\oplus \OP_Q(-3)^{\oplus m_3}.
   \]
   Then one can observe that the restriction $\Phi|_{\Spin(-1)}$ of the embedding $\Phi$ to the subbundle $\Spin(-1)\subset \K$ composed with $\mathrm{pr}$ equals $0$. Thus, by Lemma \ref{lemma: excluding the embedding of full rank} we get that either the singularities of $B$ are worse than nodes or~\mbox{$B =\{ \det(\mathrm{pr}\circ \Phi|_{\Spin(-1)}) = 0\}$}. However, the equation  $\det(\mathrm{pr}\circ \Phi|_{\Spin(-1)})$ is the section of $\OP_Q(1)$, this contradicts the fact that $B$ is the intersection of $Q$ with a quartic hypersurface. Thus, we get a contradiction and the assertion is proved.
  \end{proof}

  \begin{corollary}\label{corollary: CC-resolutions in 1/2-even case less than 28}
  Let $B$ be a nodal complete intersection of a smooth quadric hypersurface $Q$ and a quartic hypersurface in $\p^4$, let $w$ be a minimal $\frac{1}{2}$-even set of $8\mu+4$ nodes on the surface $B$. If $|w|$ is at most $20$ then~\mbox{$B = \left\{\det\left(\Phi\colon \E^{\vee}(-4)\to \E\right) = 0 \mid \Phi\in H^0(Q, (S^2\E)(5))\right\}$} where the map $\Phi$ does not split into a direct sum of two maps of proper vector subbundles  of $\E$. Moreover, only the following three cases are possible:
  \begin{align*}
   \text{\textup{(O1)}} \hspace{1cm} &|w| = 12 \text{ and } \E = \OP_Q(-1)\oplus \OP_Q(-2), && d(w) = 1;\\
   \text{\textup{(O2)}} \hspace{1cm} &|w| = 20 \text{ and } \E = \OP_Q(-2)^{\oplus 4}, && d(w) =0;\\
   \text{\textup{(O3)}} \hspace{1cm} &|w| = 20 \text{ and } \E = \Spin(-2)\oplus\OP_Q(-2)^{\oplus 4}, && d(w) =1.\\
  \end{align*}
  If one takes a general element $\Phi$ in $ H^0(Q, (S^2\E)(4))$ for $\E$ as in \textup{(O1)} or \textup{(O2)} then $ \{\det(\Phi) = 0\}$ is a nodal complete intersection of $Q$ and a quartic hypersurface in $\p^4$ with exactly $ 12$ or $20$ nodes respectively.
 \end{corollary}
 \begin{proof}
  By  assertions (3) and (5) of Lemma \ref{lemma: properties of F} and by Lemma \ref{lemma: h1(F) vanishes in case mu less than 3} we get that in our case groups $H^1(Q, \F(n))$ vanish for all $n\in \ZZ$. Then we construct the surjective map $\varphi\colon \E\to \F$ as in Lemma \ref{lemma: properties of varphi 1/2-even sets of nodes less than 28}. We denote by~\mbox{$\Phi\colon \K\to \E$} the kernel of $\varphi$ and using Lemma~\ref{lemma: excluding the isomorphism} change vector bundles $\E$ and $\K$ so that $\Phi$ is a minimal map. Therefore, by Lemma \ref{lemma: constraint of minimal map for half-even sets of nodes} we deduce $x = m_4 = 0$ and $\mu = 1$ or $2$ using the notation from \eqref{eq: E_spin in half-even case} and~\eqref{eq: E_lin in half-even case}.
  
  First, assume that $\mu = 1$. By Lemma \ref{lemma: properties of Spin}(2) we have $\Hom_Q(\Spin(-2), \OP_Q(i)) = 0$ for all $i\leqslant -3$. Thus, since the group $\Hom_Q(\Spin(-2),\Spin(-2))$ is generated by the identical map and since $\Phi$ is a minimal map we conclude that the image $\Phi\left(\Spin(-2)^{\oplus y}\oplus \OP_Q(-2)^{\oplus m_2-1}\right)$ of the subbundle $\Spin(-2)^{\oplus y})\oplus \OP_Q(-2)^{\oplus m_2-1}$ of $\K$ lies in the subbundle $\OP_Q(-1)\oplus \OP_Q(-2)^{\oplus m_2}$ of $\E$. Since the rank of the image can not decrease we deduce that~$y=0$. Analogously, the image of $\OP_Q(-2)^{\oplus m_2-1}\subset \K$ lies in $\OP_Q(-1)\subset \E$ and $m_2 = 1$ or $2$. If $m_2 = 2$ then the minimality of $\Phi$ by Lemma \ref{lemma: excluding the isomorphism} we observe that the image of the subbundle $\OP_Q(-2)\subset \K$ lies in the subbundle $\OP_Q(-1)\subset \E$ of the same rank. This is impossible by Lemma \ref{lemma: excluding the embedding of full rank}. Thus, $m_2 = 1$. Analogously, we show that $m_3 = 0$ and prove that $\K\cong \E^{\vee}(-5)$ and that $\E$ is a vector bundle as in the case $\textup{(O1)}$. 
  
  The case $\mu = 2$ can be analysed in the same way. We see that in this case $\K\cong \E^{\vee}(-5)$ and $\E$ is a vector bundle as in the cases $\textup{(O2)}$ and $\textup{(O3)}$. One can check that $H^1(Q,(S^2\E^{\vee})(-5)) = 0$ in all the cases~\mbox{\textup{(O1 -- O3)}}; thus, by Lemma~\ref{lemma: CC plus H1 vanishes implies symmetricity} the sheaf $\F$ fits the necessary symmetric Casnati--Catanese resolution. In order to compute the defects of~\mbox{$\frac{1}{2}$-ev}\-en sets of nodes  we consider the spectral sequence of cohomology groups associated with the exact sequence from Theorem \ref{theorem: exact sequence for defect} and using assertions (2) and (3) of Lemma \ref{lemma: properties of Spin} we find the necessary result. Finally, the  last assertion follows from the modification of Bertini theorem~\mbox{\cite[Lemma 4]{Barth}} and the fact that the vector bundle $(S^2\E)(5)$ is globally generated if $\E$ is the bundle as in \textup{(O1)} or \textup{(O2)}. 
 \end{proof}
 
  Note that the surface we get in the case \textup{(O2)} can be also obtained as a zero locus of a section of the vector bundle $(S^2\E)(5)$ defined in the case \textup{(O3)} if we allow $\Phi$ to split into a direct sum $\mathrm{Id}_{\Spin(-2)}\oplus \Phi'$, where~$\Phi'$ is a section of the bundle $(S^2\E')(5)$ and $\E'$ is as in the case \textup{(O2)}. However, the defects in these two cases differ, that is why we divide them in two different options. 
  Also by the next assertion there exists an easier resolution for the sheaf $\F$ in the case \textup{(O3)} of Corollary~\ref{corollary: CC-resolutions in 1/2-even case less than 28}.
 \begin{lemma}\label{lemma: case O3}
  Let $B$ be a nodal complete intersection of a smooth quadric hypersurface $Q$ and a quartic hypersurface in $\p^4$, let $w$ be a minimal $\frac{1}{2}$-even set of $20$ nodes on the surface $B$. If $w$ is as in the case \textup{(O3)} of Corollary \textup{\ref{corollary: CC-resolutions in 1/2-even case less than 28}}, then $\F$ fits into a symmetric Casnati--Catanese resolution \textup{\eqref{eq: main exact sequence}} where $\E = \Spin(-1)$. 
  
  Moreover, if one take a general element $\Phi$ in~\mbox{$ H^0(Q, (S^2\E)(5))$}, then $B = B(\Phi)$ is a nodal complete intersection of $Q$ and a quartic hypersurface in $\p^4$ and $w$ is the set of all singularities of $B$.
 \end{lemma}
 \begin{proof}
  By construction of a symmetric Casnati--Catanese resolution in Corollary \ref{corollary: CC-resolutions in 1/2-even case less than 28} and by Lemma \ref{lemma: excluding the isomorphism} the restriction of the map $\Phi$ to the subbundle $\Spin(-2)$ composed with the projection of $\E$ to the direct summand~$\Spin(-2)$ is a zero map. Denote by $\alpha\colon \Spin(-2)\to \OP_Q(-2)^{\oplus 4}$ the restriction of $\Phi$ to the subbundle~$\Spin(-1)$ composed with the projection of $\E$ to the direct summand $\OP_Q(-1)^{\oplus 4}$. Thus, we get the following commutative diagram:
\[\begin{tikzcd}[ampersand replacement=\&]
	0 \& {\Spin(-2)} \& {\Spin(-2)\oplus\OP_Q(-3)^{\oplus 4}} \& {\OP_Q(-3)^{\oplus 4}} \& 0 \\
	0 \& {\OP_Q(-2)^{\oplus 4}} \& {\Spin(-2)\oplus \OP_Q(-2)^{\oplus 4}} \& {\Spin(-2)} \& 0
	\arrow[from=1-1, to=1-2]
	\arrow[from=1-2, to=1-3]
	\arrow[from=1-3, to=1-4]
	\arrow[from=2-1, to=2-2]
	\arrow[from=2-2, to=2-3]
	\arrow[from=2-3, to=2-4]
	\arrow["\alpha"', from=1-2, to=2-2]
	\arrow["\beta", from=1-4, to=2-4]
	\arrow["\Phi"', from=1-3, to=2-3]
	\arrow[from=1-4, to=1-5]
	\arrow[from=2-4, to=2-5]
\end{tikzcd}\]
 Since $\Phi$ is a symmetric map here $\beta\colon \OP_Q(-3)^{\oplus 4}\to \Spin(-2)$ is the map dual up to a twist to $\alpha$. Thus, ranks of maps $\alpha$ and $\beta$ in all points of $Q$ are equal. Moreover, since $\alpha$ is an embedding of sheaves then in a general point of $Q$ ranks of $\alpha$ and $\beta$ equal 2.
 
 By Lemma \ref{lemma: map from o to spin} the support of the cokernel of $\beta$ is either empty or it is a line in the quadric threefold~\mbox{$Q\subset \p^4$} or it is a hyperplane section of $Q$. Since the surface $B$ is irreducible and the rank of $\F$ is greater than $1$ only along a finite set $w$ by the snake lemma we get that the cokernel of $\beta$ vanishes. Thus, the rank of the map~$\beta$  equals~$2$ in all points of $Q$ and so does the rank of $\alpha$.
 
 Therefore, $\alpha$ is an embedding of vector bundles. Thus, it coinsides with the map~\mbox{$\Spinmap\colon \Spin(-2)\to \OP_{Q}(-2)^{\oplus 4}$} described in Lemma \ref{lemma: properties of Spin}. Then the cokernel of the map~$\alpha$ equals $\Spin(-1)$ and the kernel of the map $\beta$ equals~$\Spin(-3)$. Thus, by the snake lemma we get that $\F$ fits into the following exact sequence:
  \[
   0\to \Spin(-3)\xrightarrow{\Phi'} \Spin(-1)\to \F\to0.
  \]
  Thus, by Lemma \ref{lemma: CC plus H1 vanishes implies symmetricity} we get a new symmetric Casnati--Catanese resolution for the sheaf $\F$. Finally, since the vector bundle~$(S^2\E')(5)$ is globally generated then by~\mbox{\cite[Lemma 4]{Barth}} we get the last assertion of the lemma. 
 \end{proof}

 \section{Minimal $\frac{1}{2}$-even sets of 28 nodes}\label{section: 1/2-even sets of 28 nodes}
 Our goal in this section is to prove Proposition \ref{proposition: 28 nodes implies 12 nodes}. The proof of this assertion is quite long, here we give its sketch. We fix a nodal surface $B$ in a smooth quadric threefold $Q$ as in the proposition with a minimal $\frac{1}{2}$-even set of $28$ nodes $w$. We denote by~$\F$ the sheaf associated to $w$ on $Q$ and construct a symmetric Casnati--Catanese resolution for $\F$. The construction of the resolution involves several steps. First, we consider a double covering $p\colon Q\to \p^3$ and construct a symmetric Casnati--Catanese resolution of the sheaf $p_*\F$ on $\p^3$. This helps us to deduce new properties of the sheaf~$\F$ and to construct a free resolution of $\F$ on $Q$ which is not symmetric though. However, using properties of the sheaf $\F$ deduced from this resolution we finally construct the symmetric Casnati--Catanese resolution for $\F$. This allows us to find a hyperplane in $\p^4$ which is tangent to $B$ along a conic and intersect $\Sing(B)$ in a $0$-even set $w'$ of $12$ nodes. Thus, we conclude that the surface $B$ lies in the family \textup{(E1)} from Theorem \ref{theorem: classification of minimal sets of nodes}.
 \subsection{Pushforward of $\F$ to $\p^3$}\label{subsection: pushforward of F in case of 28 nodes}
 Choose a point $x\in \p^4\setminus Q$ and let $p\colon \p^{4}\dashrightarrow \p^3$ be the linear projection from the point $x$ to a hyperplane in $\p^4$. Then the restricted map $p\colon Q\to \p^3$ is a finite morphism of degree~$2$. Consider the direct image  $p_*\F$ of the sheaf $\F$ to $\p^3$. This sheaf has the following property.
 \begin{lemma}\label{lemma: Grothendieck duality after pushforward to p3}
  Let $B$ be a nodal complete intersection of a smooth quadric hypersurface $Q$ and a quartic hypersurface in $\p^4$, let $w$ be a $\frac{1}{2}$-even set of nodes on $B$ and let $\F$ be the sheaf on $Q$ with support on $B$ associated  with $w$ on $Q$. If $p\colon Q\to \p^3$ is a double cover induced by a projection from $\p^4$ to a hyperplane then one has the following isomorphism for any $m\in \ZZ$:
  \[
   \ExtSh^1_{\p^3}(p_*\F(m), \OP_{\p^3}) \cong  p_*\F(6-m).
  \]
 \end{lemma}
 \begin{proof}
 Since the relative canonical bundle $\omega_{Q/\p^3}$ equals~$\OP_Q(1)$ then using the Grothendieck duality as stated in \cite[Theorem III.11.1]{Hartshorne_Grothendieck_duality} we deduce the following:
  \[
   \ExtSh^1_{\p^3}(p^*\F(m),\OP_{\p^3}) \cong p_* \ExtSh_Q^1(\F(m),\OP_Q(1))\cong p_*\ExtSh_Q^1(\F(m-1),\OP_Q).
  \]
  Finally, by Lemma~\mbox{\ref{lemma: properties of F}$(2)$} one has $\ExtSh_Q^1(\F(m-1),\OP_Q)\cong \F(6-m)$. Thus, the proof is complete.
 \end{proof}
 The cohomology groups of the direct image $p_*\F$ are related to those of the sheaf $\F$. Namely, the following assetion holds.
 \begin{lemma}\label{lemma: cohomologies of F for 28 nodes}
  Let $B$ be a nodal complete intersection of a smooth quadric hypersurface $Q$ and a quartic hypersurface in $\p^4$, let $w$ be a minimal $\frac{1}{2}$-even set of $28$ nodes on $B$ and let $p\colon Q\to \p^3$ be a double cover induced by a projection from $\p^4$ to a hyperplane. If $\F$ is the sheaf on $Q$ with support on $B$ associated  with~$w$ then
  \[
   H^j(\p^3,p_*\F(n))\cong H^j(Q, \F(n))
  \]
  for all $0\leqslant j\leqslant 3$ and all $n\in \ZZ$. Moreover, $H^1(\p^3,p_*\F(1))\cong \CC^2$ and $ H^0(\p^3,p_*\F(1)) = 0$ and $p_*\F(3)$ is globally generated. 
 \end{lemma}
 \begin{proof}
  The isomorphism $H^j(\p^3,p_*\F(n))\cong H^j(Q, \F(n))$ holds because $R^i p_*\F(n) = 0$ for all $i\geqslant 1$ and for all integer $n$.
  
  Analogously to the proof of Lemma \ref{lemma: h1(F) vanishes in case mu less than 3} one can show that~\mbox{$h^0(Q,\F(1)) = 0$}.  Thus, by Lemma \ref{lemma: properties of F})(3) we have $h^2(\p^3,p_*\F(1)) = h^0(\p^3,p_*\F(1)) =0$ and $\chi(p_*\F(1)) = -2$. By Lemma \ref{lemma: properties of F})(7) one has $\chi(\F(1)) = -2$, and therefore, $h^1(\p^3,p_*\F(1))=0$. Finally, by Lemma \ref{lemma: properties of F}(5) and (6)  cohomology groups $H^i(\p^3,p_*\F(3-i))$ vanish for all $1\leqslant i \leqslant 3$; thus, $p_*\F(3)$ is globally generated by \cite[Lecture 14]{Mumford}. 
 \end{proof}
 Let $\F_B$ be the sheaf associated to $w$ on $B$. Then $\F = \iota_*\F_B$ and $h^0(Q, \F(1)) = h^0(B,\F_B(1))$ since the closed embedding $\iota\colon B\to Q$ is an affine morphism. Then by Serre duality on the surface $B$ there is an isomorphism $\Ext_B^1(\F_B (1),\OP_B(K_B))\cong H^1(B,\F_B(1))$ where $K_B$ is the canonical class on the surface $B$. This isomorphism and the cup product induce the following perfect pairing:
 \begin{equation}\label{eq: Serre pairing 28 nodes}
 H^1(\p^3,p_*\F(1))\otimes H^1(\p^3,p_*\F(1)) \cong  H^1(B,\F(1))\otimes H^1(B,\F(1)) 
  \to H^2(B,\OP_B(K_B)) \cong \CC.
 \end{equation}
 The last isomorphism is due to Serre duality on the surface $B$. This pairing is altrenating. Denote by $U$ a one-dimensional isotropic subspace of $H^1(\p^3,p_*\F(1))$ with respect to this pairing. Then by~\cite{Walter} there exists a map of sheaves on $\p^3$
 \[
  \varphi_{\Omega}\colon \Omega^1_{\p^3}(-1)\to p_*\F,
 \]
 which induces  embeddings between the first cohomology groups of the twist of the sheaves  $\Omega^1_{\p^3}(-1)$ and~$p_*\F$ by $\OP_{\p^3}(2)$, and moreover,~\mbox{$H^1(\varphi_{\Omega})(H^1(\p^3,\Omega^1_{\p^3})) = U\subset H^1(\p^3,p_*\F(1))$.}
 
 Denote by $A$ the graded algebra $\bigoplus_{n\in \ZZ} H^0(\p^3,\OP_Q(n))$.
 Then $M = \bigoplus_{n\in \ZZ} H^0(\p^3,\F(n))/ \Im(\varphi_{\Omega}(\ZZ))$ is a graded left $A$-module; here by $\Im(\varphi_{\Omega}(\ZZ))$  the following graded submodule:
 \[
  \Im(\varphi_{\Omega}(\ZZ)) = \bigoplus_{n\in \ZZ}\left( H^0(\varphi_{\Omega}(n-1)) H^0(\p^3,\Omega^1_{\p^3}(n-1))\right)\subset \bigoplus_{n\in \ZZ} H^0(Q,\F(n)).
 \]
 Since the cohomology group $H^0(Q,\F(n))$ vanishes for all $n\leqslant 0$ and by Lemma \ref{lemma: cohomologies of F for 28 nodes} the module $M$ is generated by elements in its first, second and third graded components. 
 
 Choose the minimal set of graded generators of $M$ (i.e. any element from the set of generators lies in a graded component of $M$) and assume that this set contains exactly $m_i$ elements of $H^0(\p^3,p_*\F(i))$ for~$i = 1,2$ and $3$. Denote by $\E_{lin}$ the vector bundle $\OP_Q(-1)^{\oplus m_1}\oplus \OP_Q(-2)^{\oplus m_2}\oplus \OP_Q(-3)^{\oplus m_3}$. By Lemmas \ref{lemma: properties of F} and \ref{lemma: cohomologies of F for 28 nodes} we get $m_1 = 0$ and $m_2 = 2$. Then the set of generators we have choosen defines a map from $\E_{lin}$ to~$\F$:
 
 \[
  \varphi_{lin}\colon \E_{lin}\to \F.
 \]
 Denote by $\E$ and $\varphi$ the vector bundle $\Omega^1_{\p^3}(-1)\oplus \E_{lin}$ and the map $\varphi = \varphi_{\Omega}\oplus\varphi_{lin}\colon \E\to \F$. The construction of the vector bundle $\E$ and the sheaf map $\varphi$ has the following properties.
 \begin{lemma}\label{lemma: properties of varphi 1/2-even sets of 28 nodes}
  Let $\E = \Omega^1_{\p^3}(-1)\oplus \OP_Q(-2)^{\oplus 2}\oplus \OP_Q(-3)^{\oplus m_3}$ and let the map $\varphi\colon \E\to p_*\F$ be as defined above. Then $\varphi$ is surjective and satisfies the following properties:
  \begin{enumerate}
   \item[\textup{(1)}] the map $H^0(\varphi(n))\colon H^0(\p^3,\E(n)) \to H^0(\p^3,p_*\F(n))$ is surjective for all $n\in \ZZ$ and this map is an isomorphism for~\mbox{$n= 2$};
   \item[\textup{(2)}] the map $H^1(\varphi(n))\colon H^1(\p^3,\E(n)) \to H^1(\p^3,p_*\F(n))$ is an embedding to the isotropic subspace $U$ of the space $H^1(\p^3,p_*\F(1))$ for $n=1$.
  \end{enumerate}
 \end{lemma}
 \begin{proof}
  One has $\varphi = \varphi_{\Omega}\oplus\varphi_{lin} \colon \E = \Omega^1_{\p^3}(-1)\oplus\E_{lin}\to p_*\F$ where $\E_{lin}$ and $\varphi_{\Omega}\oplus\varphi_{lin}$ are the vector bundle and the map constructed above.
  The assertion~\textup{(2)} holds by construction of the $\varphi_{\Omega}$ and since the vector bundle $\E_{lin}$ has no intermediate cohomology groups. The assertion \textup{(1)} holds since we choose the vector bundle $\E_{lin}$ in the way that the map~\mbox{$H^0(\p^3,\E_{lin}(n)) \to H^0(\p^3,p_*\F(n))/ H^0(\varphi_{\Omega}(n-1))(H^0(\p^3,\Omega^1_{\p^3}(n-1)))$} is surjective for all $n\in\ZZ$. 
  
  Finally, the fact that $\varphi$ is sujective follows from assertion \textup{(1)}.
 \end{proof}

 Now we show that $p_*\F$ fits into a symmetric Casnati--Catanese resolution, this is the main result of Section~\ref{subsection: pushforward of F in case of 28 nodes}.
 \begin{proposition}[{c.f. \cite[Theorem 0.3]{Casnati_Catanese}}]\label{proposition: resolution for direct image of F in case of 28 nodes}
   Let $B$ be a nodal complete intersection of a smooth quadric hypersurface $Q$ and a quartic hypersurface in $\p^4$, let $w$ be a minimal $\frac{1}{2}$-even set of $28$ nodes on the surface~$B$ and let $p\colon Q\to \p^3$ be a double cover induced by a projection from $\p^4$ to a hyperplane. Then the sheaf $p_*\F$ fits into the following exact sequence
   \[
    0\to\E^{\vee}(-6)\xrightarrow{\Phi} \E\xrightarrow{\varphi} p_*\F\to 0,
   \]
   where $\E$ is the vector bundle $\Omega^1_{\p^3}(-1)\oplus\OP_{\p^3}(-2)^{\oplus 2}$ and $\Phi\in H^0(\p^3,(S^2\E)(6))\subset \Hom_{\p^3}(\E^{\vee}(-6),\E)$. 
 \end{proposition}
 \begin{proof}
 The proposition follows from the construction in the proof of \cite[Theorem 0.3]{Casnati_Catanese}. We recall here the main steps of the proof for our particular case.
 
 Consider the kernel of the map $\varphi\colon \E\to p_*\F$ constructed in Lemma \ref{lemma: properties of varphi 1/2-even sets of 28 nodes}. It is a vector bundle since the projective dimension of $p_*\F$ equals 1. One can easily check that $h^1(\p^3,\Ker(\varphi)(n))$ vanishes for all $n\in\ZZ$ and
 \[
  H^2(\p^3,\Ker(\varphi)(n)) = \left\{\begin{aligned}
                                & U^{\vee}, \text{ if }n = 2;\\
                                & 0, \text{ otherwise.}
                               \end{aligned}\right.
 \]
 The intermediate cohomology groups of all twists of a given vector bundle define it up to line summands by~\mbox{\cite[Theorem 0.4]{Walter}}. Note that the vector bundle $T_{\p^3}(-5)$ has the same intermediate cohomology groups as~$\Ker(\varphi)$. Therefore, we get that~$\Ker(\varphi)$ is the vector bundle of the following form:
 \begin{equation*}
  \Ker(\varphi) = T_{\p^3}(-5)\oplus \bigoplus_{i\geqslant 3}\OP_{\p^3}(-i)^{\oplus n_i}.
 \end{equation*}
 where  $n_i$ are non-negative integers for $i\geqslant 3$. By Lemma \ref{lemma: properties of varphi 1/2-even sets of nodes less than 28} we get that $m_4 = n_3 = 0$, $n_4 = 2$ and $n_i = 0$ for all $i\geqslant 5$. Thus, we get a vector bundle $\E$ is as in the assertion and $\Ker(\varphi)\cong \E^{\vee}(-6)$. The existence of the symmetric map between the vector bundles $\Ker(\varphi)$ and $\E$ such that the cokernel equals $p_*\F$ follows from~\mbox{\cite[Section 2]{Casnati_Catanese}}.
 \end{proof}
 The main result of this step of the proof of Proposition \ref{proposition: 28 nodes implies 12 nodes} is the fact that $\varphi\colon \Omega^1_{\p^3}(-1)\oplus\OP_{\p^3}(-2)^{\oplus 2}\to p_*\F$ is a surjective map. In particular, we will not use that fact that the map $\Phi$ is symmetric. However, the simplest way to prove this fact is to construct a resolution using methods developed in~\cite{Walter} and \cite{Casnati_Catanese}; thus, the symmetricity of the map $\Phi\colon \E^{\vee}(-6)\to \E$ follows immediately.

 \subsection{A non-symmetric resolution of $\F$}\label{subsection: non-symmetric resolution}
 Consider the Euler exact sequence for the cotangent bundle on the projective space~$\p^4$:
 \[
  0\to \Omega_{\p^4}^1\to \OP_{\p^4}(-1)^{\oplus 5}\to \OP_{\p^4}\to 0.
 \]
 The restriction of this exact sequence to the hypersurface  $Q$ remains a short exact sequence. Twist this exact sequence by $\OP_Q(-1)$ and apply the functor $\Hom_Q(-, \F)$ to it. We get the following long exact sequence.
 \[
  \dots \to \Hom_Q(\Omega^1_{\p^4}(-1)|_Q,\F)\xrightarrow{\delta} \Ext_Q^1(\OP_Q(-1),\F)\to \Ext^1_Q(\OP_Q(-2)^{\oplus 5}, \F)\to\dots
 \]
 Since $\Ext^1_Q(\OP_Q(-2)^{\oplus 5}, \F) \cong H^1(Q,\F(2))^{\oplus 5} = 0$ by Lemma \ref{lemma: properties of F}(2) the connecting homomorphism $\delta$ is surjective. Thus, since $\Ext_Q^1(\OP_Q(-1),\F)\cong H^1(Q,\F(1)) = \CC^2$ there exists the following map
 \[
  \alpha'\colon \Omega^1_{\p^4}(-1)|_Q^{\oplus 2} \to \F,
 \]
 such that $H^1*(\alpha'(n))\colon H^1(Q,\Omega^1_{\p^4}(n-1)|_Q^{\oplus 2})\to H^1(Q,\F(n))$ is an isomorphism for $n = 1$.
 This is also an isomorphism for all $n\in\ZZ\setminus\{1\}$ since by Lemma \ref{lemma: properties of F}(5) and (6) one has $h^1(Q,\F(n)) = h^1(Q,\Omega_{\p^4}^1(n-1)) = 0$. Denote by $\alpha''$ the map 
 \[
  \alpha''\colon \OP_Q(-2)^{\oplus 2}\to \F.
 \]
 from the vector bundle $\OP_Q(-2)^{\oplus 2}$ to $\F$ such that $H^0(\alpha''(2))\colon H^0(Q,\OP_Q)\to H^0(Q,\F(2))$ is an isomorphism. Denote by $\M$ the vector bundle $\Omega_{\p^4}^1(-1)|_Q^{\oplus 2} \oplus \OP_Q(-2)^{\oplus 2}$ and by  $\alpha = \alpha'\oplus \alpha''\colon \M \to \F$ the direct sum of two maps defined above.  
 \begin{lemma}\label{lemma: surjectivity of alpha for 28 nodes}
  Let $B$ be a nodal complete intersection of a smooth quadric hypersurface $Q$ and a quartic hypersurface in $\p^4$, let $w$ be a minimal $\frac{1}{2}$-even set of $28$ nodes on the surface~$B$ and let $\M$ be a vector bundle~$\Omega_{\p^4}^1(-1)|_Q^{\oplus 2} \oplus \OP_Q(-2)^{\oplus 2}$. Then the map of sheaves~\mbox{$\alpha\colon \M \to \F$} constructed above is surjective.
 \end{lemma}
 \begin{proof}
  Let $p\colon Q\to \p^3$ be a double cover induced by a projection from $\p^4$ to a hyperplane. One can check  using the Euler exact sequence that $p_*\M = \Omega_{\p^3}^1(-1)^{\oplus 2}\oplus\OP_{\p^3}(-2)^{\oplus 12} \oplus \OP_{\p^3}(-3)^{\oplus 2}$.  
  Consider the direct image of the map $\alpha$ to $\p^3$:
  \[
   p_*\alpha\colon p_*\M \to p_*\F.
  \]
  Our goal is to show that the map $\varphi$ constructed in Proposition \ref{proposition: resolution for direct image of F in case of 28 nodes} factors through $p_*\alpha$.

Since $p$ is a finite morphism, one has $H^i(\p^3,p_*\M(n))\cong H^i(Q, \M(n))$. Recall that the vector space $U$ is an isotropic subspace of $H^1(Q,\F(1))$ for the perfect alternating pairing \eqref{eq: Serre pairing 28 nodes}. By construction of the bundle~$\M$ we have~\mbox{$H^1(\p^3,\M(1)) \cong H^1(Q,\F(1)) = U\oplus U^{\vee}$}. Therefore, there exists a map~\mbox{$\beta'\colon p^*\Omega_{\p^3}^1(-1)\otimes U\to \M$} which induces an embedding of the first cohomology groups of all twists, one can construct it by the Euler exact sequence or by \cite[Theorem 0.4]{Walter}. 
  
  Consider the map $p_*\alpha\colon p_*\M(2)\to p_*\F(2)$ and the multiplication map $s\colon \OP_{\p^3} \otimes H^0(Q,\F(2)) \to p_*\F(2)$. Both these maps are defined by surjective maps on zeroth cohomology groups.  Thus, there exists a map of vector bundles $\beta''\colon \OP_{\p^3}(-2) \otimes H^0(Q,\F(2)) \to p_*\M$ such that $s = p_*\alpha\circ\beta''$. 
  
  Denote by $\beta$ the direct sum of maps of vector bundles $\beta = \beta'\oplus\beta''\colon \Omega_{\p^3}^1(-1) \oplus \OP_{\p^3}(-2)^{\oplus 2}\to p_*\M$. Then we deduce that $\beta$ is an embedding and that the following diagram commutes:
\[\begin{tikzcd}[ampersand replacement=\&]
	{\Omega_{\p^3}^1(-1)^{\oplus 2}\oplus\OP_{\p^3}(-2)^{\oplus 12} \oplus \OP_{\p^3}(-3)^{\oplus 2}} \& {p_*\F} \\
	{\Omega_{\p^3}^1(-1)\oplus \OP_{\p^3}(-2)^{\oplus 2}} \& {p_*\F}
	\arrow["\varphi", from=2-1, to=2-2]
	\arrow["{p_*\alpha}", from=1-1, to=1-2]
	\arrow["{\beta}"', hook, from=2-1, to=1-1]
	\arrow[no head, equal, from=2-2, to=1-2]
\end{tikzcd}\]
 Here the map $\varphi$ is as in Proposition \ref{proposition: resolution for direct image of F in case of 28 nodes}. Thus, $\varphi$ is a surjective map of sheaves then so is $p_*\alpha$; and therefore, so is $\alpha$. This finishes the proof.
 \end{proof}

 In the next assertion we construct a non-symmetric resolution of the sheaf $\F$.
 \begin{lemma}\label{lemma: non-symmetric exact sequence for 28 nodes}
 Let $B$ be a nodal complete intersection of a smooth quadric hypersurface $Q$ and a quartic hypersurface in $\p^4$, let $w$ be a minimal $\frac{1}{2}$-even set of $28$ nodes on the surface~$B$. Then there exists the following resolution for the sheaf $\F$:
 \begin{equation}\label{eq: non-symmetric exact sequence for 28 nodes}
  0\to \Spin(-2)^{\oplus 4}\oplus \OP_{Q}(-3)^{\oplus 2}\to \Omega_{\p^4}^1(-1)|_Q^{\oplus 2} \oplus \OP_Q(-2)^{\oplus 2} \xrightarrow{\alpha} \F\to 0.
 \end{equation}
\end{lemma}
\begin{proof}
 By Lemma \ref{lemma: surjectivity of alpha for 28 nodes} the map $\alpha$ is surjective. Thus, $\Ker(\alpha)$ is a vector bundle since the projective dimension of $\F$ equals $1$. Moreover, by construction of the map $\alpha$ we have
 \[
  H^i(Q,\Ker(\alpha)(n)) = 0,
 \]
 for all $i = 1$ or $2$ and all $n\in\ZZ$. Thus, by \cite[Theorem 3.5]{Ottaviani-Horrocks_criterion} this implies that $\Ker(\alpha)$ is a direct sum of several twists of spinor bundles and line bundles. Since $\alpha(n)$ induces the surjective maps on zeroth cohomology groups for all $n\in \ZZ$ one can compute:
 \begin{align*}
  \rk(\Ker(\alpha)) = 10; && h^0(Q, \Ker(\alpha)(2)) = 0; && h^0(Q, \Ker(\alpha)(3)) = 18.
 \end{align*}
 These computations imply the result.
 \end{proof}
 Lemma \ref{lemma: non-symmetric exact sequence for 28 nodes} implies the following condition on cohomology groups of the sheaf $\F$.
 \begin{corollary}\label{corollary: cohomology groups of F otimes Spin for 28 nodes}
 Let $B$ be a nodal complete intersection of a smooth quadric hypersurface $Q$ and a quartic hypersurface in $\p^4$, let $w$ be a minimal $\frac{1}{2}$-even set of $28$ nodes on the surface~$B$. Then
 \[
  h^i(Q, \F\otimes \Spin(2)) = \left\{ \ \begin{aligned}
                                &4, \text{ if } i = 1;\\
                                &0, \text{ otherwise.}
                               \end{aligned}\right.
 \]
\end{corollary}
 \begin{proof}
  The result follows from the exact sequence \eqref{eq: non-symmetric exact sequence for 28 nodes}.
 \end{proof}

 \subsection{Symmetric Casnati--Catanese resolution for $\F$}
 Recall that $U$ is a one-dimensional isotropic subspace of $H^1(Q,\F(1))$ with respect to the pairing \eqref{eq: Serre pairing 28 nodes}. By the Euler exact sequence there exists a map of sheaves:
 \[
  \psi'\colon \Omega_{\p^4}^1(-1)|_Q \to \F.
 \]
 which induces the embedding onto the subspace $U\subset H^1(Q,\F(1))$ on the first cohomology groups of the twist of sheaves~\mbox{$\Omega_{\p^4}^1(-1)|_Q$} and $\F$ by $\OP_Q(1)$. Denote by $\psi''\colon \OP_Q(-2)^{\oplus 2}\to \F$ the map of sheaves such that~\mbox{$H^0(\psi''(2))\colon H^0(Q,\OP_Q^{\oplus 2})\to H^0(Q,\F(2))$} is an isomorphism. 
 Set $\E= \Omega_{\p^4}^1(-1)|_Q\oplus  \OP_Q(-2)^{\oplus 2}$ and let $\psi$ be the direct sum of two maps:
 \begin{equation}\label{eq: surjective map for CC resolution for 28 nodes}
  \psi = \psi'\oplus\psi''\colon \E = \Omega_{\p^4}^1(-1)|_Q\oplus  \OP_Q(-2)^{\oplus 2} \to \F.
 \end{equation}
 We use here the same notation as in Section \ref{subsection: pushforward of F in case of 28 nodes} for a different vector bundle. However, since in this section we never refer to that one we hope that this abuse of notation does not lead to a misunderstanding. Note also that the map $\psi''$ coincides with the map $\alpha''$ constructed in Section \ref{subsection: non-symmetric resolution}. Then the map $\psi$ has the following properties.
 \begin{lemma}\label{lemma: properties of psi in case of 28 nodes}
 Let $B$ be a nodal complete intersection of a smooth quadric hypersurface $Q$ and a quartic hypersurface in $\p^4$, let $w$ be a minimal $\frac{1}{2}$-even set of $28$ nodes on the surface~$B$ and let $\psi\colon \E \to \F$ be the map \textup{\eqref{eq: surjective map for CC resolution for 28 nodes}}. Then
 \begin{enumerate}
  \item[\textup{(1)}] The map $\psi$ is surjective.
  \item[\textup{(2)}]  The map $H^0(\psi(2))\colon H^0(Q,\E(2))\to H^1(Q,\F(2))$ is an isomorphism.
  \item[\textup{(3)}]  The map $H^1(\psi(1))\colon H^1(Q,\E(1))\to H^1(Q,\F(1))$ is an embedding to the subspace $U\subset H^1(Q,\F(1))$.
  \item[\textup{(4)}] There is an embedding $\gamma\colon \E\to \Omega_{\p^4}^1(-1)|_Q^{\oplus 2} \oplus \OP_Q(-2)^{\oplus 2}$ such that $\psi = \alpha\circ\gamma$ where $\alpha$ is the map defined as in Lemma  \textup{\ref{lemma: surjectivity of alpha for 28 nodes}}. 
 \end{enumerate}
 \end{lemma}
 \begin{proof}
  The proof of the surjectivity of $\psi$ is analogous to the proof of Lemma \ref{lemma: properties of varphi 1/2-even sets of 28 nodes}. The second and third assertions follow from the construction of $\psi$. The map $\gamma$ exists since the constructions of $\alpha$ and $\psi$ are analogous.
 \end{proof}
 Denote by $\Psi\colon\K\to\E$ the kernel of the surjective map of sheaves $\psi$. Thus, we get the following short exact sequence:
 \begin{equation}\label{eq: exact sequence for 28 nodes}
  0\to \K\xrightarrow{\Psi} \E\xrightarrow{\psi}\F\to 0.
 \end{equation}
 The sheaf $\K$ has the following properties.
 \begin{lemma}\label{lemma: orthogonality of kernel of psi to exceptional objects for 28 nodes}
 Let $B$ be a nodal complete intersection of a smooth quadric hypersurface $Q$ and a quartic hypersurface in $\p^4$, let $w$ be a minimal $\frac{1}{2}$-even set of $28$ nodes on the surface~$B$ and let $\K$ be the kernel of the map $\psi$ defined in \textup{\eqref{eq: surjective map for CC resolution for 28 nodes}}. Then $\K$ is vector bundle and for all $ 0\leqslant i\leqslant 3$ one has the following equalities:
 \begin{align*}
  &\Ext_Q^i(\OP_Q(-2), \K) = 0;\\
  &\Ext_Q^i(\Spin(-1), \K) = 0.
 \end{align*}
\end{lemma}
\begin{proof}
 Apply the functor $ \Hom_Q( \OP_Q(-2), -)$ to the exact sequence \eqref{eq: exact sequence for 28 nodes}. We get the following long exact sequence of derived functors:
 \[
  0\to \Hom_Q(\OP_Q(-2), \K) \to H^0(Q, \Omega_{\p^4}^1(1)|_Q \oplus \OP_Q^{\oplus 2}) \to H^0(Q, \F(2))\to \Ext^1_Q(\OP_Q(-2), \K)\to \dots
 \]
 By Lemma \ref{lemma: properties of psi in case of 28 nodes} the map of cohomology groups $H^0(Q, \Omega_{\p^4}^1(1)|_Q \oplus \OP_Q^{\oplus 2}) \to H^0(Q, \F(2))$ is an isomorphism. Since $H^i(Q,\E(2)) = H^i(Q,\F(2)) = 0$ for all $i = 1,2$ and $3$, groups $\Ext_Q^i(\OP_Q(-2), \K)$ vanish for all $ 0\leqslant i\leqslant 3$.
 
 Also by Lemma \ref{lemma: properties of psi in case of 28 nodes} there is the embedding $\gamma\colon \E\to \Omega_{\p^4}^1(-1)|_Q^{\oplus 2} \oplus \OP_Q(-2)^{\oplus 2}$ such that $\alpha = \psi\circ\gamma$. Then $\gamma$ induces the map $\gamma_{\Ker}\colon \K \to \Spin(-2)^{\oplus 4}\oplus \OP_{Q}(-3)^{\oplus 2}$ between kernels of maps $\psi$ and $\alpha$. Since $\gamma$ is an embedding then so is $\gamma_{\Ker}$.
 
 Apply the functor $\Hom_Q(\Spin(-1), -)$ to both exact sequences \eqref{eq: non-symmetric exact sequence for 28 nodes} and \eqref{eq: exact sequence for 28 nodes}. The map $\gamma$ induces the following commutative diagram:
\[\begin{tikzcd}[ampersand replacement=\&]
	{H^1(Q, \E\otimes\Spin(2))} \&\& {H^1(Q,\F\otimes \Spin(2))} \\
	{H^1(Q, (\Omega_{\p^4}^1|_Q(-1)^{\oplus 2}\oplus \OP_Q(-2)^{\oplus 2}) \otimes\Spin(2))} \&\& {H^1(Q,\F\otimes \Spin(2))}
	\arrow["{H^1(\gamma \otimes \mathrm{Id}_{\Spin(2)})}"', from=1-1, to=2-1]
	\arrow["{H^1(\psi \otimes \mathrm{Id}_{\Spin(2)})}", from=1-1, to=1-3]
	\arrow["{H^1(\alpha \otimes \mathrm{Id}_{\Spin(2)})}"', from=2-1, to=2-3]
	\arrow[no head, equal, from=1-3, to=2-3]
\end{tikzcd}\]
The map $H^1(\alpha \otimes \mathrm{Id}_{\Spin(2)})\colon H^1(Q, \Omega_{\p^4}^1|_Q \otimes\Spin(1)^{\oplus 2}\oplus \Spin^{\oplus 2})) \to H^1(Q,\F\otimes \Spin(2))$ is surjective; this follows from Lemma~\ref{lemma: non-symmetric exact sequence for 28 nodes} and the fact that $H^2(Q,(\Spin(-2)^{\oplus 4}\oplus \OP_Q(-3)^{\oplus 2})\otimes\Spin(2)) = 0$.
Then the commutativity of the diagram implies that the map~\mbox{$\psi \otimes \mathrm{Id}_{\Spin(2)}$} is also surjective; thus, it is an isomorphism. Then the result follows by Corollary~\ref{corollary: cohomology groups of F otimes Spin for 28 nodes}.
\end{proof}
Now we finally can construct a symmetric Casnati--Catanese resolution for the sheaf $\F$.
\begin{corollary}\label{corollary: symmetric exact sequence for 28 nodes}
 Let $B$ be a nodal complete intersection of a smooth quadric hypersurface $Q$ and a quartic hypersurface in $\p^4$ and let $w$ be a minimal $\frac{1}{2}$-even set of $28$ nodes on the surface~$B$. Then the sheaf $\F$ associated to $w$ on $Q$ fits into a symmetric Casnati--Catanese resolution \textup{\eqref{eq: main exact sequence}} where $\E = \Omega_{\p^4}^1(-1)|_Q \oplus \OP_Q^{\oplus 2}(-2)$ and $d(w) = 1$.
\end{corollary}
\begin{proof}
 The proof of this assertion uses derived categories of coherent sheaves on varieties and the notion of full exceptional collections in them. We refer to \cite{Huybrechts} for all necessary definitions and facts.

 Denote by $D^b(Q)$ the bounded derived category of the category of coherent sheaves on the smooth quadric threefold~$Q$. Consider the following full exceptional collection in $D^b(Q)$:
 \[
  \langle \OP_Q(-3), T_{\p^4}(-4)|_Q, \OP_Q(-2), \Spin(-1)\rangle.
 \]
 This collection arises after several left mutations of the standard collection, see \cite[Section 4]{Kapranov}.
 By Lemma \ref{lemma: orthogonality of kernel of psi to exceptional objects for 28 nodes} we get that the kernel $\K$ of the map $\psi$ defined in \eqref{eq: surjective map for CC resolution for 28 nodes} is orthogonal to elements $\OP_Q(-2)$ and~$\Spin(-1)$ of this collection. Thus, $\K$ lies in the subcategory of the category~$D^b(Q)$ generated by vector bundles $\OP_Q(-3)$ and $T_{\p^4}(-4)|_Q$. Since $\Ext_Q^1(\OP_Q(-3), T_{\p^4}(-4)|_Q)$ vanishes this implies that $\K$ is a direct sum of several copies of these two bundles. Computing the cohomology groups of $\K$ from the exact sequence~\eqref{eq: exact sequence for 28 nodes} we conclude that it is as follows:
 \[
  \K =  T_{\p^4}(-4)|_Q \oplus \OP_Q(-3)^{\oplus 2} =  \E^{\vee}(-5).
 \]
 Now it remains to show that $\Psi$ can be replaced by an element from the subgroup $H^0(Q, (S^2\E)(5))$ of the group~$\Hom(\E^{\vee}(-5), \E)$. This follows from Lemma \ref{lemma: CC plus H1 vanishes implies symmetricity} and the fact that $H^1(Q,(S^2\E^{\vee})(5)) = 0$. This finishes the construction of a symmetric Casnati--Catanese resolution. The fact that $d(w) = 0$ follows from Theorem~\ref{theorem: exact sequence for defect}.
\end{proof}

\subsection{Proof of Proposition \textup{\ref{proposition: 28 nodes implies 12 nodes}}}
In this section we study properties of symmetric maps between vector bundles $T_{\p^4}(-4)|_Q$ and $\Omega_{\p^4}^1(-1)|_Q$. Then we use these results in order to prove that there is no minimal $\frac{1}{2}$-even sets of 28 nodes on a complete intersection of a smooth quadric hypersurface and a quartic hypersurface in $\p^4$.

By \cite[Example 1.5]{Ottaviani-spinors} one has $\Omega_Q^1\cong S^2\Spin$. Since the normal bundle of $Q$ in $\p^4$ equals $\OP_Q(2)$ we see that bundles $\Omega_{\p^4}^1(-1)|_Q$ and $T_{\p^4}(-4)|_Q$ fit into the following exact sequences:
 \begin{equation}\label{eq: normal and conormal sequences for Q}
 \begin{split}
  0\to \OP_Q(-3)\to \Omega_{\p^4}^1(-1)|_Q&\xrightarrow{a} (S^2\Spin)(-1)\to 0;\\
  0\to (S^2\Spin)(-2) \xrightarrow{b}T_{\p^4}(-4)|_Q &\to \OP_Q(-2)\to 0.
  \end{split}
 \end{equation}
 Given a map from $\Omega_{\p^4}^1(-1)|_Q$ to $T_{\p^4}(-4)|_Q$ we can consider the map from~$(S^2\Spin)(-2)$ to $(S^2\Spin)(-1)$ which arises as the composition with maps in these exact sequences. 
\begin{lemma}\label{lemma: special hyperplane in case of 28 nodes}
 Let $\Xi\colon  T_{\p^4}(-4)|_Q\to \Omega_{\p^4}^1(-1)|_Q $ be a map of sheaves on a smooth quadric threefold $Q$. Then there exists a hyperplane section $\Pi$ of $Q$ such that the composition $a\circ \Xi\circ b$ vanishes along $\Pi$.
\end{lemma}
\begin{proof}
 The dimension of the space $\Hom_Q((S^2\Spin)(-2),(S^2\Spin)(-1))$ equals $5$. On the other hand, one can construct a~\mbox{$5$-di}\-mensional space of maps from $(S^2\Spin)(-2)$ to $(S^2\Spin)(-1)$; namely, to each section $l\in H^0(Q,\OP(1))$ one can associate a multiplication of a local section of $(S^2\Spin)(-2)$ by the linear form $l$. Therefore, all elements of the group~\mbox{$\Hom_Q((S^2\Spin)(-2),(S^2\Spin)(-1))$} can be constructed this way.
 
 The composition $a\circ \psi\circ b$ lies in  $\Hom_Q((S^2\Spin)(-2),(S^2\Spin)(-1))$; thus, it is a multiplication by a linear form~$l$. Then it vanishes along the hyperplane section $\Pi = \{l = 0\}$.
\end{proof}

Now we are ready to prove Proposition \ref{proposition: 28 nodes implies 12 nodes}.
\begin{proof}[Proof of Proposition \textup{\ref{proposition: 28 nodes implies 12 nodes}}]
 Assume that there exists a nodal surface $B$ which is a complete intersection of a smooth quadric threefold $Q$ and a quartic hypersurface in $\p^4$ such that $B$ contains a minimal $\frac{1}{2}$-even set of $28$ nodes. Then by Corollary \ref{corollary: symmetric exact sequence for 28 nodes} one can construct a symmetric Casnati--Catanese resolution for the sheaf~$\F$ associated to $w$ on $Q$. More precisely, take $\E = \Omega_{\p^4}^1(-1)\oplus \OP_Q(2)$ then there exists $\Psi\in H^0(Q,(S^2\E)(5))$ such that~$B = B(\Psi)$ and $w = w(\Psi)$. 
 
  Let $a$ and $b$ be maps of sheaves defined in~\eqref{eq: normal and conormal sequences for Q}. Set $a'\colon \E\to  (S^2\Spin)(-1)$ and  $b'\colon (S^2\Spin)(-2)\to \E^{\vee}(-5)$ be the direct sums of $a$ and $b$ respectively with the zero maps. 

  By Lemma \ref{lemma: special hyperplane in case of 28 nodes} there exists a hyperplane section $\Pi\subset Q$ such that the composition $a'\circ\Psi\circ b'$ equals zero along $\Pi$. Thus, after a restriction to $\Pi$ there exist elements $\phi\in \Hom_Q((S^2\Spin)(-2)|_{\Pi}, \OP_{\Pi}(-3)\oplus \OP_{\Pi}(-2)^{\oplus 2})$ and $\psi \in\Hom_Q( \OP_{\Pi}(-2)\oplus \OP_{\Pi}(-3)^{\oplus 2}))$ such that the following diagram commutes:
\[\begin{tikzcd}[ampersand replacement=\&]
	0 \& {(S^2\Spin)(-2)|_{\Pi}} \& {T_{\p^4}|_{\Pi}(-4)\oplus \OP_{\Pi}(-3)^{\oplus 2}} \& {\OP_{\Pi}(-2)\oplus \OP_{\Pi}(-3)^{\oplus 2}} \& 0 \\
	0 \& {\OP_{\Pi}(-3)\oplus \OP_{\Pi}(-2)^{\oplus 2}} \& {\Omega_{\p^4}^1|_{\Pi}(-1)\oplus \OP_{\Pi}(-2)^{\oplus 2}} \& {(S^2\Spin)(-1)|_{\Pi}} \& 0
	\arrow[from=1-1, to=1-2]
	\arrow["{b'}", from=1-2, to=1-3]
	\arrow[from=1-3, to=1-4]
	\arrow[from=2-1, to=2-2]
	\arrow[from=2-2, to=2-3]
	\arrow["{a'}", from=2-3, to=2-4]
	\arrow[from=1-4, to=1-5]
	\arrow[from=2-4, to=2-5]
	\arrow["\Psi|_{\Pi}", from=1-3, to=2-3]
	\arrow["\phi"', from=1-2, to=2-2]
	\arrow["{\psi}", from=1-4, to=2-4]
\end{tikzcd}\]
Since the morphism $\Psi|_{\Pi}$ is symmetric one has $\phi = \psi^{\vee}(-5)$. Since $B$ is an irreducible surface of degree 8 in~$\p^4$, the morphism $\Psi$ is injective in a general point of the hyperplane $\Pi$. Therefore, maps $\psi$ and $\phi$ are also morphisms of sheaves of rank 3 in a general point of $\Pi$. Therefore, kernels of $\phi$, $\Psi$ and $\psi$ vanish and by the snake lemma we have an exact sequence of cokernels:
\[
 0\to \Coker(\phi) \to \F|_{\Pi} \to \Coker(\psi)\to 0.
\]
Since maps $\phi$ and $\psi$ are dual up to a twist their ranks decrease on a same scheme $C = \{\det(\phi) = 0\}$. Note that $C$ can not be empty since vector bundles ${(S^2\Spin(-2))|_{\Pi}}$ and $\OP_{\Pi}(-3)\oplus \OP_{\Pi}(-2)^{\oplus 2}$ are not isomorphic. Moreover, $\dim(C) = 1$. Thus, ranks of sheaves $\Coker(\phi)$ and $\Coker(\psi)$ are at least $1$ along $C$. Since the rank of the sheaf $\F$ increases in a finite set of points we deduce that the plane $\Pi$ is tangent to $B$ along the curve $C$.    

Then the curve $C$ is an intersection of two quadric surfaces in $\p^3$; thus, the genus of $C$ is at most 1. On the other hand, $C$ is isomorphic to its proper preimage~$\hat{C}$ in the minimal resolution of singularities $\hat{\delta}\colon\hat{B}\to B$ of the surface $B$. Let $w' = \Sing(B)\cap C$. Then by adjunction formula the degree of the canonical class $K_{\hat{C}}$ of the curve $\hat{C}$ on $\hat{B}$ is as follows:
\[
 \deg(K_{\hat{C}})= (\hat{\delta}^*K_B + {\hat{C}})\cdot {{\hat{C}}} = \left(\frac{3H - \sum_{p\in w'}E_p}{2}\right)\cdot\left(\frac{H - \sum_{p\in w'}E_p}{2}\right) = \frac{24 - 2|w'|}{4}.
\]
Here we denote by $E_p$ the exceptional divisor in $\hat{B}$ for any point $p\in \Sing(B)$. Since the divisor $H - \sum_{p\in w'}E_p$ lies in $2\Pic(\hat{B})$ the set of nodes $w'$ is a $\frac{1}{2}$-even.  Thus, the formula and Corollary \ref{corollary: CC-resolutions in 1/2-even case less than 28} implies that  $w'$ consists of 12 nodes. Finally, $w'$ does not intersect $w$ since the ranks of  sheaves $\Coker(\phi)$ and~$\Coker(\psi)$ equals $1$ in all points of $C$.
\end{proof}

\section{Proofs of Theorem \ref{theorem: classification of minimal sets of nodes} and Corollary \ref{corollary: obstructions to rationality}}\label{section: proofs}
 
 We are ready to prove the classification of minimal $\frac{\delta}{2}$-even sets of nodes on a nodal complete intersection of a smooth quadric threefold and a quartic threefold in $\p^4$.
\begin{proof}[Proof of Theorem \textup{\ref{theorem: classification of minimal sets of nodes}}]
Let $B$ be a nodal complete intersection of a smooth quadric hypersurface $Q$ and a quartic hypersurface in $\p^4$ and let $w$ be a minimal $\frac{\delta}{2}$-even set of nodes on $B$. If $\delta = 0$ then by Corollary \ref{corollary: CC-resolutions in 0-even case} there exists a vector bundle $\E$ on $Q$ and an element $\Phi\in H^0(Q,(S^2\E)(4))$ such that $B = B(\Phi)$ and $w = w(\Phi)$ and $\E$ one of the vector bundles as in cases \textup{(E1 -- E3)}. Thus, Theorem \ref{theorem: classification of minimal sets of nodes} is proved in the case of $0$-even sets of nodes.

Now assume that $w$ is the minimal $\frac{1}{2}$-even set of nodes on $B$. Then by Lemma \ref{lemma: properties of F} the order of the set of nodes $w$ is~$|w| = 8\mu + 4$. If $|w|\geqslant 36$ then by Lemma \ref{lemma: properties of F}(9) we get that $w$ is not minimal. If $B$ contains a minimal $\frac{1}{2}$-even set of $28$ nodes then it also contains a minimal $\frac{1}{2}$-even set $w'$ of $12$ nodes. Thus, we are in the case \textup{(O1)}. Finally if  $|w|\leqslant 20$ then by Corollary \ref{corollary: CC-resolutions in 1/2-even case less than 28} and Lemma \ref{lemma: case O3} we get the cases \textup{(O1 -- O3)} and the proof of the theorem is finished.
\end{proof}
Now we establish three families of double covers of a smooth quadric ramified over a nodal surface of degree 8 which admit Artin--Mumford obstructions to rationality.
\begin{proof}[Proof of Corollary \textup{\ref{corollary: obstructions to rationality}}]
 Assume that $\pi\colon \doublecover\to Q$ is a double cover of a smooth quadric threefold $Q$ ramified over a nodal surface $B$ which is a complete intersection of $Q$ with a quartic threefold in $\p^4$. If $\doublecover$ has an Artin--Mumford obstruction to rationality, then by Corollary \ref{corollary: AM implies w with 0 defect} there exists a non-empty $\frac{\delta}{2}$-even set of nodes $w$ on $B$ such that the defect of $w$ equals $0$. By Lemma \ref{lemma: defect of a subset} we can assume that $w$ is minimal. Theorem~\ref{theorem: classification of minimal sets of nodes} and Corollary \ref{corollary: symmetric exact sequence for 28 nodes} implies that there exists a vector bundle $\E$ and an element~\mbox{$\Phi$} in the cohomology group~\mbox{$H^0(Q,(S^2\E)(4+\delta))$} such that~$B = B(\Phi)$. Moreover, the vector bundle $\E$ is as in cases~\mbox{\textup{(E1 -- E3)}} or~\mbox{\textup{(O1 -- O3)}}, or in the case $|w| = 28$ the bundle $\E$ is as in Corollary \ref{corollary: symmetric exact sequence for 28 nodes}. Then using Theorem~\ref{theorem: exact sequence for defect} we compute the defects of $\frac{\delta}{2}$-even sets of nodes in each case. The only cases when the defect of $w$ vanishes are~\mbox{\textup{(E2), (E3)} and \textup{(O2)}}. Finally, by~\cite[Lemma 4]{Barth} we get the second assertion of Corollary \ref{corollary: obstructions to rationality}. This finishes the proof.
\end{proof}

\bibliographystyle{alpha}
\bibliography{even_sets}

\end{document}